\documentclass{amsart}
\usepackage[english,frenchb]{babel}
\usepackage{lmodern}
\usepackage{enumerate}
\usepackage{amsmath}
\usepackage{amsfonts}
\usepackage{amssymb}
\usepackage{amsthm}
\usepackage{graphicx}
\usepackage{mathrsfs}
\usepackage{pxfonts}
\usepackage{datetime}
\usepackage{comment}
\usepackage{fourier-orns}
\usepackage{dsfont}
\usepackage[left=3cm,right=3cm,top=3cm,bottom=3cm]{geometry}
\usepackage[cyr]{aeguill}
\usepackage{fancyhdr}
\usepackage{mathabx}
\usepackage[all,cmtip]{xy}
\usepackage[latin1]{inputenc}
\usepackage{amscd}
\usepackage{enumitem}
\usepackage{hyperref}

\usepackage{tikz-cd}
\usetikzlibrary{decorations.pathmorphing}

\usepackage{amsfonts,amsmath,graphicx}
\newcommand{\bc}{\begin{center}}
\newcommand{\ec}{\end{center}}
\newcommand{\bq}{\begin{quote}}
\newcommand{\eq}{\end{quote}}
\newcommand{\bqtn}{\begin{quotation}}
\newcommand{\eqtn}{\end{quotation}}
\newcommand{\beq}{\begin{equation}}
\newcommand{\eeq}{\end{equation}}
\newcommand{\bearr}{\begin{eqnarray}}
\newcommand{\eearr}{\end{eqnarray}}
\newcommand{\bearrn}{\begin{eqnarray*}}
\newcommand{\eearrn}{\end{eqnarray*}}
\newcommand{\bi}{\begin{itemize}}
\newcommand{\ei}{\end{itemize}}
\newcommand{\be}{\begin{enumerate}}
\newcommand{\ee}{\end{enumerate}}
\newcommand{\bthe}{\begin{theorem}}
\newcommand{\ethe}{\end{theorem}}
\newcommand{\blem}{\begin{lemme}}
\newcommand{\elem}{\end{lemme}}
\newcommand{\bsolu}{\begin{solution}}
\newcommand{\esolu}{\end{solution}}
\newcommand{\bexer}{\begin{exercise}}
\newcommand{\eexer}{\end{exercise}}

\newcommand{\Aut}{{\rm Aut}}

\newcommand{\Tr}{{\rm Tr}}

\newcommand{\RR}{\mathbb{R}}
\newcommand{\QQ}{\mathbb{Q}}
\newcommand{\ZZ}{\mathbb{Z}}

\newcommand{\FF}{\mathbb{F}}

\newcommand{\AAA}{\mathbb{A}} 

\newcommand{\ba}{\begin{array}}
\newcommand{\ea}{\end{array}}

\newcommand{\mfp}{\mathfrak{p}}
\newcommand{\mfP}{\mathfrak{P}}
\newcommand{\mfq}{\mathfrak{q}}
\newcommand{\mfQ}{\mathfrak{Q}}

\newcommand{\mfM}{\mathfrak{M}}

\newcommand{\Char}{\operatorname{char}}

\newcommand{\im}{\operatorname{im}}
\usepackage{amsfonts,amsmath}
\newtheorem{theoreme}{Theorem}[section]
\newtheorem{theorem}[theoreme]{Theorem}
\newtheorem{lemme}[theoreme]{Lemma}
\newtheorem{lemma}[theoreme]{Lemma}
\newtheorem{proposition}[theoreme]{Proposition}
\newtheorem{definition}[theoreme]{Definition}

\newtheorem{corollaire}[theoreme]{Corollary}

\newtheorem{solution}[theoreme]{Solution}
\newtheorem{exercise}[theoreme]{Exercise}
\newtheorem{question}[theoreme]{Question}
\newtheorem{innercustomthm}{Proposition}
\newenvironment{customthm}[1]
  {\renewcommand\theinnercustomthm{#1}\innercustomthm}
  {\endinnercustomthm}
\newcommand{\bdefi}{\begin{definition}}
\newcommand{\edefi}{\end{definition}}
\newcommand{\brk}{\begin{remarque}}
\newcommand{\erk}{\end{remarque}}
\newcommand{\bpp}{\begin{proposition}}
\newcommand{\epp}{\end{proposition}}
\newcommand{\bpf}{\begin{proof}}
\newcommand{\epf}{\end{proof}}
\newcommand{\bcor}{\begin{corollaire}}
\newcommand{\ecor}{\end{corollaire}}
\newcommand{\bsol}{\begin{solution}}
\newcommand{\esol}{\end{solution}}
\theoremstyle{definition}

\newtheorem{remarque}[theoreme]{Remark}
\newtheorem{remark}[theoreme]{Remark}
\allowdisplaybreaks
\title{When are permutation invariants Cohen-Macaulay over all fields?}
\author{Ben Blum-Smith and Sophie Marques}
\email{ben@cims.nyu.edu}
\email{sophie.marques@uct.ac.za}

\begin{document}
\large
\selectlanguage{english}
\maketitle

\begin{abstract}
We prove that the polynomial invariants of a permutation group are Cohen-Macaulay for any choice of coefficient field if and only if the group is generated by transpositions, double transpositions, and 3-cycles. This unites and generalizes several previously known results. The ``if" direction of the argument uses Stanley-Reisner theory and a recent result of Christian Lange in orbifold theory. The ``only-if" direction uses a local-global result based on a theorem of Raynaud to reduce the problem to an analysis of inertia groups, and a combinatorial argument to identify inertia groups that obstruct Cohen-Macaulayness. 
\end{abstract}

\section{Introduction}

The invariant ring of a graded action by a finite group $G$ on a polynomial ring 
\[
k[\mathbf{x}]=k[x_1,\dots,x_n]
\]
over a field $k$ is well-behaved when the field characteristic is prime to the group order. For example, it is generated in degree $\leq |G|$ (Noether's bound), and it is a Cohen-Macaulay ring (the Hochster-Eagon theorem).

When the characteristic divides the group order (the {\em modular case}), the situation is much more mysterious. Both of these statements (and many others) can, but do not always, fail. The question of when such pathologies arise has attracted research attention over the last few decades.

In this article we focus on Cohen-Macaulayness. Let $k[\mathbf{x}]^G$ be the invariant ring and let 
\[
p = \operatorname{char}k
\]
 be the field characteristic. We interpret $k[\mathbf{x}]$ as the coordinate ring of $\AAA_k^n$, so that the action of $G$ on $k[\mathbf{x}]$ is induced from an action on $\AAA_k^n$ by automorphisms. Because the action on $k[\mathbf{x}]$ is graded, the corresponding action on $\AAA_k^n$ is linear, i.e. it arises from a linear representation of $G$ on a $k$-vector space. Here is a sampling of known results:

\begin{itemize}
\item In 1980, Ellingsrud and Skjelbred (\cite{ellingsrud}) showed that if $G$ is cyclic of order $p^m$, then $k[\mathbf{x}]^G$ is not Cohen-Macaulay unless $G$ fixes a subspace of $\AAA_k^n$ of codimension $\leq 2$. 
\item In 1996, Larry Smith (\cite{smith96}) showed that if $n = 3$, then $k[\mathbf{x}]^G$ is Cohen-Macaulay. (This was priorly known to hold for $n\leq 2$.)
\item In 1999, Campbell et al (\cite{campbelletal}) showed that if $G$ is a $p$-group, and if the action of $G$ on $\AAA_k^n$ is the sum of three copies of the same linear representation, then $k[\mathbf{x}]^G$ is not Cohen-Macaulay. 
\item Also in 1999, Gregor Kemper (\cite{kemper99}) showed that if $G$ is a $p$-group and $k[\mathbf{x}]^G$ is Cohen-Macaulay, then $G$ is necessarily generated by elements $g$ whose fixed-point sets in $\AAA_k^n$ have codimension $\leq 2$, generalizing \cite{ellingsrud} beyond cyclic groups and \cite{campbelletal} beyond three-copies representations.
\end{itemize}

See \cite{kemper12} for a more detailed overview. 

A theme uniting these results is that generation of $G$ by elements fixing codimension $\leq 2$ subspaces is related to good behavior of $k[\mathbf{x}]^G$. Further variations on this theme are found in \cite{dufresne}, \cite{gordeevkemper}, \cite{kacwatanabe}, and \cite{lorenzpathak}. The main goal of this paper is a result of this kind for permutation groups $G\subset S_n$, acting on $k[\mathbf{x}]$ by permuting the $x_i$'s. The result characterizes permutation groups generated in this way, and is not restricted to $p$-groups. 

Permutation groups have the feature that the definition of the action is insensitive to the choice of a ground field $k$. Thus it is natural to ask: 

\begin{question}\label{q:CMpermutationgroups}
For which $G\subset S_n$ is $k[\mathbf{x}]^G$ Cohen-Macaulay regardless of $k$?
\end{question}

An additional motivation for this question is that $k[\mathbf{x}]^G$ is Cohen-Macaulay for every choice of $k$ if and only if $\ZZ[\mathbf{x}]^G$ is free as a module over the subring $\ZZ[\mathbf{x}]^{S_n}$ of symmetric polynomials, and also if and only if $A[\mathbf{x}]^G$ is Cohen-Macaulay for every Cohen-Macaulay ring $A$. (We will not develop these equivalences here, but see \cite[\S 2.4.1]{blumsmith} where the first is worked out in detail, and \cite[Exercise 5.1.25]{brunsherzog} for a sketch of the second in a slightly different setting.) 

In \cite{kemper01}, Kemper gave an if-and-only-if criterion that determines Cohen-Macaulayness of a permutation invariant ring when $p$ divides $|G|$ exactly once. This criterion allows to determine Cohen-Macaulayness for many specific groups and primes, but does not in general answer question \ref{q:CMpermutationgroups} because few permutation groups have squarefree order. Some special cases of question \ref{q:CMpermutationgroups} are known:

\begin{itemize}
\item If $G$ is a Young subgroup (i.e. a product of symmetric groups acting on disjoint sets), then $k[\mathbf{x}]^G$ is a polynomial algebra over $k$, so it is Cohen-Macaulay regardless of $k$.

\item It follows from the result of Kemper (\cite{kemper99}) quoted above that if $G$ is a $p$-group, then $k[\mathbf{x}]^G$ cannot be Cohen-Macaulay over all fields unless $G$ is generated by transpositions and double transpositions, or 3-cycles (and $p=2$ or $3$).

\item Kemper also showed in (\cite{kemper99}) that if $G\subset S_n$ is regular (i.e. its action on 
\[
[n] = \{1,\dots,n\}
\]
is free and transitive), then $k[\mathbf{x}]^G$ is Cohen-Macaulay over every $k$ if it is isomorphic to $C_2$, $C_3$, or $C_2\times C_2$, but not otherwise. (In fact, in other cases, it is not Cohen-Macaulay for any $k$ with $\operatorname{char}k$ dividing $|G|$.)

\item Victor Reiner (\cite{reiner92}, \cite{hersh2}) has shown that $A_n$, and the diagonally embedded $S_n\hookrightarrow S_n\times S_n\subset S_{2n}$, have invariant rings that are Cohen-Macaulay regardless of the field. (These are the $S_n$-cases of results he found for all finite Coxeter groups.) Patricia Hersh (\cite{hersh}, \cite{hersh2}) has shown the same for the wreath product $S_2\wr S_n\subset S_{2n}$.
\end{itemize} 

Our main objective in this article is to answer question \ref{q:CMpermutationgroups} completely. We will prove the following theorem, which unites all of these cases and ties them into the theme mentioned above.

\begin{theorem}\label{mainresult}
Let $G\subset S_n$. The ring $k[\mathbf{x}]^G$ is Cohen-Macaulay for all choices of $k$ if and only if $G$ is generated by transpositions, double transpositions, and 3-cycles.
\end{theorem}

Let $N$ be the subgroup of $G$ generated by transpositions, double transpositions, and 3-cycles. The ``if" direction of theorem \ref{mainresult}, together with the Hochster-Eagon theorem (\cite[Proposition 13]{hochstereagon}), imply that the characteristics $p$ in which $k[\mathbf{x}]^G$ fails to be Cohen-Macaulay must be among those that divide $[G:N]$. This implication will be discussed in more detail in the conclusion (\S \ref{sec:concl}). The ``only-if" direction implies that if $[G:N]>1$, then there is at least one such characteristic $p$. This $p$ is explicitly constructed in the course of the proof.

The proof of this theorem is methodologically eclectic. The ``if" direction uses Stanley-Reisner theory, which relates Cohen-Macaulayness of $k[\mathbf{x}]^G$ to the topology of the quotient of a ball by $G$, and a recent result in orbifold theory by Christian Lange (\cite{lange2}) that characterizes the groups $G$ such that this quotient is a piecewise-linear ball. The ``only-if" direction is much more algebraic. It is based on a local-global result (theorem \ref{m}) reducing the Cohen-Macaulayness of a noetherian invariant ring to that of the invariant rings of its inertia groups acting on strict localizations.

Though theorem \ref{mainresult} is specific to the situation of a polynomial ring $k[\mathbf{x}]$ and a permutation group $G$, a substantial portion of our method for the ``only-if" direction applies in considerably more generality. Section \ref{sec:groupactions} concerns arbitrary commutative, unital rings, and the local-global result just mentioned only assumes that the invariant ring is noetherian. (Other work on Cohen-Macaulayness of invariants at the generality of noetherian rings includes \cite{gordeevkemper} and \cite{lorenzpathak}.) A secondary goal of this paper is to develop these general tools, which we expect have broader applicability. The fact that Cohen-Macaulayness depends fully on the local action of the inertia groups yields information about Cohen-Macaulayness whenever inertia groups can be accessed directly and are simpler than the whole group, as in the present case.

The method of the ``if" direction is similar to the methods used by Reiner and Hersh (\cite{reiner92}, \cite{hersh}, \cite{hersh2}) to prove the results mentioned above. The novelty is the application of Lange's orbifold result (\cite{lange2}) in place of an explicit shelling of a cell complex. The main novelties in the ``only-if" direction are the local-global theorem \ref{m}; its application to show that certain kinds of inertia $p$-groups obstruct Cohen-Macaulayness (proposition \ref{prop:pobstruction}); and a combinatorial argument that exhibits such an inertia $p$-group explicitly in the case at hand (lemma \ref{lem:primestabilizer}).

The organization of the paper is as follows. Section \ref{sec:background} collects together the needed background from commutative algebra, Stanley-Reisner theory, and piecewise-linear topology, and introduces notation that is used throughout the article. Section \ref{sec:inertiaCM} contains the general results on Cohen-Macaulayness and inertia groups that are needed for the ``only-if" direction of theorem \ref{mainresult}, including the local-global theorem \ref{m} and the $p$-group obstruction proposition \ref{prop:pobstruction}. Section \ref{perminvars} proves the ``if" direction of theorem \ref{mainresult}, and then using this, proves the ``only-if" direction. Finally, section \ref{sec:concl} draws out some implications and poses questions for further inquiry.

\section{Background}\label{sec:background}

Throughout this paper, $A$ denotes an arbitrary commutative, unital ring, $k$ denotes a field, $p$ denotes the characteristic of $k$, $k[\mathbf{x}]$ denotes the polynomial ring $k[x_1,\dots,x_n]$, $[n]$ denotes the set $\{1,\dots,n\}$, and $G$ denotes a finite group with a faithful action on $k[\mathbf{x}]$ by permutations of the $x_i$'s, or on $A$ by arbitrary automorphisms. In \S \ref{sec:onlyif}, the prime number $p$ will be conceptually prior to $k$, and $k$ will be chosen to satisfy $\Char k = p$.

\subsection{Cohen-Macaulayness}\label{sec:CMness}

Recall that the \textbf{depth} of a local noetherian ring is the length of the longest regular sequence contained in the maximal ideal. The depth is always bounded above by the dimension. When equality is achieved, the ring is said to be \textbf{Cohen-Macaulay}. A general noetherian ring is defined to be Cohen-Macaulay if its localization at every maximal, or equivalently at every prime, is Cohen-Macaulay (\cite[Definition 2.1.1 and Theorem 2.1.3(b)]{brunsherzog}).

Although there has been work on extending the theory of Cohen-Macaulayness to the non-noetherian setting (\cite{hamiltonmarley}), in this paper we will follow tradition by regarding noetherianity as a requirement of Cohen-Macaulayness.

Cohen-Macaulayness is automatic for artinian rings, since if the dimension is zero, the depth of a localization cannot be strictly lower than this. For example, fields are Cohen-Macaulay. Noetherian regular rings, for example polynomial rings over fields, are also Cohen-Macaulay (\cite[Corollary 2.2.6]{brunsherzog}).

For our purposes it will be necessary to know how the Cohen-Macaulayness of a ring relates to that of a flat extension. The needed fact (\cite[Theorem 2.1.7]{brunsherzog}) is that if $A\rightarrow B$ is a flat extension of noetherian rings, then $B$ is Cohen-Macaulay if and only if, for each prime ideal $\mfq$ of $B$ and its contraction $\mfp$ in $A$, both $A_\mfp$ and $B_\mfq/\mfp B_\mfq$ are Cohen-Macaulay. It is enough to quantify this statement over maximal ideals $\mfq$ of $B$. We will use this fact repeatedly in \S \ref{sec:inertiaCM}.


When a noetherian ring is finite over a regular subring, Cohen-Macaulayness is related to flatness as a module over the subring. In the traditional situation of invariant theory, this fact has a particularly nice formulation. For if $k[\mathbf{x}]$ is a polynomial ring over a field, and $G$ acts by graded automorphisms, then $k[\mathbf{x}]^G$ is finitely generated and graded, and the Noether normalization lemma guarantees a graded polynomial subring (generated by a {\em homogeneous system of parameters}) over which $k[\mathbf{x}]^G$ is finite. In this situation, $k[\mathbf{x}]^G$ is Cohen-Macaulay if and only if it is a {\em free} module over this subring (the {\em Hironaka criterion}). We will not build on this fact directly, but we mention it both because it motivates interest in Cohen-Macaulayness, and because we do use a result (\cite[Theorem A.1]{hersh2}) that depends on it, whose proof we outline in the next section.

\subsection{Combinatorial commutative algebra and PL topology}\label{sec:CCAandPL}

The proof of the ``if" direction of theorem \ref{mainresult} relies on results in combinatorial commutative algebra and some basic facts about PL topology. For motivation, we describe the plan of the proof before recalling these results.

By work of Adriano Garsia and Dennis Stantion \cite{garsiastanton}, refined by Victor Reiner in \cite{hersh2}, Cohen-Macaulayness of the polynomial invariant ring $k[\mathbf{x}]^G$ can be deduced from the Cohen-Macaulayness of the {\em Stanley-Reisner ring} of a certain cell complex (specifically a {\em boolean complex}) that depends on $G$. The Cohen-Macaulayness of this Stanley-Reisner ring can in turn be deduced from information about the complex that depends only on the homeomorphism class of its total space. For $G$ generated as in theorem \ref{mainresult}, a recent result of Christian Lange \cite{lange2} hands us this topological information. This is the structure of the proof, which will be assembled in section \ref{sec:if}. Here, we recall the needed results and definitions regarding boolean complexes and Stanley-Reisner rings.

Let $P$ be a finite poset and $k$ a field.

\begin{definition}\label{def:SRofposet}
The \textbf{Stanley-Reisner ring} of $P$ over $k$, written $k[P]$, is the quotient of the polynomial ring $k[\{y_\alpha\}_{\alpha\in P}]$, with indeterminates indexed by the elements of $P$, by the ideal generated by products $y_\alpha y_\beta$ indexed by incomparable pairs $\alpha,\beta\in P$.
\end{definition}

\begin{remark}
This is a special case of a more general definition, which we will not use directly: the Stanley-Reisner ring of a simplicial complex. (We will use a further generalization -- see definition \ref{def:SRofboolean} below.) The Stanley-Reisner ring of a poset is nothing but the Stanley-Reisner ring of the {\em chain complex} of the poset, i.e. the simplicial complex with vertex set the elements of the poset, whose simplices are the chains in the poset. It is helpful to keep in mind that the Stanley-Reisner ring of a poset has an underlying simplicial complex as well.
\end{remark}

Write $[n] = \{1,\dots,n\}$. Let $B_n$ be the boolean algebra on the set $[n]$, i.e. the set of subsets of $[n]$, ordered by inclusion. Then the Stanley-Reisner ring $k[B_n\setminus\{\emptyset\}]$ is, in a sense that can be made precise, a coarse approximation of the polynomial ring $k[\mathbf{x}]$. In particular, it carries a natural action of $S_n$ via the latter's action on the set $[n]$, and if $G\subset S_n$, then $k[\mathbf{x}]^G$ is Cohen-Macaulay whenever $k[B_n\setminus\{\emptyset\}]^G$ is Cohen-Macaulay. This is the content of \cite[Theorem A.1]{hersh2}. 

The proof is given in full there, and also in great detail in \cite[Section 2.5.3]{blumsmith}, and in any case is essentially a characteristic-neutral reformulation of an argument of Adriano Garsia and Dennis Stanton in \cite{garsiastanton}, building on Garsia's earlier work \cite{garsia}. However, we would like this result to be better-known, so we indicate the line of proof.

As mentioned in section \ref{sec:CMness}, a finitely generated graded $k$-algebra is Cohen-Macaulay if and only if it is free as a module over the subring generated by any homogeneous system of parameters. Thus, Cohen-Macaulayness can be established by showing the existence of a module basis over such a subring. For any $G\subset S_n$, $k[\mathbf{x}]^{S_n}$ and $k[B_n\setminus\{\emptyset\}]^{S_n}$ are such subrings, respectively, of $k[\mathbf{x}]^G$ and $k[B_n\setminus\{\emptyset\}]^G$, and they are isomorphic. Thus, Cohen-Macaulayness may be passed from $k[B_n\setminus\{\emptyset\}]^G$ to $k[\mathbf{x}]^G$ by showing that the existence of a module basis for the former over the common subring $k[B_n\setminus\{\emptyset\}]^{S_n} \cong k[\mathbf{x}]^{S_n}$ implies the existence of a basis for the latter. In \cite{garsia}, Garsia introduced a $k$-linear, $S_n$-equivariant map $\mathscr{G}: k[B_n\setminus\{\emptyset\}] \rightarrow k[\mathbf{x}]$ sending
\[
y_U \mapsto \prod_{i\in U} x_i,
\]
where $U\in B_n\setminus\{\emptyset\}$ is any nonempty subset of $[n]$. The map $\mathscr{G}$ is first extended multiplicatively to all monomials of $k[B_n\setminus\{\emptyset\}]$, and then $k$-linearly to the whole ring. This map is an isomorphism of $k$-vector spaces, and also, in a sense made precise in \cite[Proposition 2.5.66]{blumsmith}, a coarse approximation of a ring homomorphism. In particular, for any $G\subset S_n$, if $k[B_n\setminus\{\emptyset\}]^G$ is Cohen-Macaulay, it maps an appropriately chosen $k[B_n\setminus\{\emptyset\}]^{S_n}$-basis of $k[B_n\setminus\{\emptyset\}]^G$ to a $k[\mathbf{x}]^{S_n}$-basis of $k[\mathbf{x}]^G$. This statement about bases was proven by Garsia and Stanton in \cite{garsiastanton} with $k=\QQ$, in which case both rings are automatically Cohen-Macaulay -- Garsia and Stanton's interest was in the explicit construction of bases -- but it was observed by Reiner in \cite[Theorem A.1]{hersh2} that the argument is characteristic-neutral and so allows one to deduce Cohen-Macaulayness of $k[\mathbf{x}]^G$ from that of $k[B_n\setminus\{\emptyset\}]^G$ in the modular situation.

\begin{remark}\label{rmk:transfer}
Garsia \cite{garsia}, Garsia-Stanton \cite{garsiastanton}, and Reiner \cite{hersh2} all refer to the map $\mathscr{G}$ as the {\em transfer map}. Other authors in invariant theory (\cite{neuselsmith}, \cite{smith95}) use the same phrase to denote the $A^G$-linear map
\begin{align*}
\Tr : A &\rightarrow A^G\\
x &\mapsto \sum_{g\in G} g(x).
\end{align*}
While this latter map is also called the {\em trace}, there are well-established usages of {\em transfer} to describe maps analogous to $\Tr$ in both topology and group theory, so we prefer to call $\mathscr{G}$ the {\em Garsia map} to avoid competition for the term and to honor Garsia's introduction of it in \cite{garsia}. The present paper makes no use of the Garsia map except implicitly in quoting \cite[Theorem A.1]{hersh2}.
\end{remark}

The work cited above reduces proving Cohen-Macaulayness of $k[\mathbf{x}]^G$ to the analogous statement for $k[B_n\setminus\{\emptyset\}]^G$. The Cohen-Macaulayness of this latter ring can be assessed using a topological criterion, following a general philosophy in Stanley-Reisner theory that the Cohen-Macaulayness of a Stanley-Reisner ring is equivalent to a condition on the homology of the underlying simplicial complex. In the present situation, $k[B_n\setminus\{\emptyset\}]^G$ is not the Stanley-Reisner ring of a poset or simplicial complex, but it turns out to be the Stanley-Reisner ring of a {\em boolean complex}. We recall the needed definitions:

\begin{definition}\label{def:booleancomplex}
A \textbf{boolean complex} is a regular CW complex in which every face has the combinatorial type of a simplex.
\end{definition}
This is a mild generalization of a simplicial complex, in which it is possible for two faces to intersect in an arbitrary subcomplex rather than a single subface. (For example, two faces can have all the same vertices.) See figure \ref{fig:booleancomplex}. The terminology is due to Garsia and Stanton in \cite{garsiastanton}.

The \textbf{face poset} of a cell complex is the poset whose elements are the cells (\textbf{faces}), and the relation $\alpha \leq \beta$ means that $\alpha$'s closure is contained in $\beta$'s closure. For our purposes it is convenient to modify this definition to include an additional \textbf{empty face} $\emptyset$, with $\emptyset \leq \alpha$ for all faces $\alpha$. With this convention, a boolean complex can be characterized as a regular CW complex whose face poset has the property that every lower interval is a finite boolean algebra; this is the etymology of the name {\em boolean complex}. Face posets of boolean complexes are referred to as {\em simplicial posets}, a term introduced by Richard Stanley in \cite{enumcomb}.


\begin{figure}
\begin{center}
\begin{tikzpicture}

\node at (-1.3,0) {A};
\draw [fill=black] (-1,0) circle [radius = 0.1];
\node at (1.3,0) {B};
\draw [fill=black] (1,0) circle [radius = 0.1];

\draw (-1,0) arc [radius = 1, start angle = 180, end angle = 360];
\node at (0,-1.2) {C};
\draw (1,0) arc [radius = 1, start angle = 0, end angle = 180];
\node at (0,1.2) {D};

\node at (5,-1.5) {$\emptyset$};
\node at (3.9,0) {A};
\node at (6.1,0) {B};
\node at (3.9,1) {C};
\node at (6.1,1) {D};

\tikzstyle{every node} = [draw, shape=circle]

\node (empt) at (5,-1) {};
\node (A) at (4.3,0) {};
\node (B) at (5.7,0) {};
\node (C) at (4.3,1) {};
\node (D) at (5.7,1) {};

\foreach \from/\to in {empt/A, empt/B, A/C, A/D, B/C, B/D}
    \draw (\from) -- (\to);

\end{tikzpicture}
\end{center}
\caption{Left: a boolean complex with total space homeomorphic to a circle. Right: its face poset.}\label{fig:booleancomplex}
\end{figure}
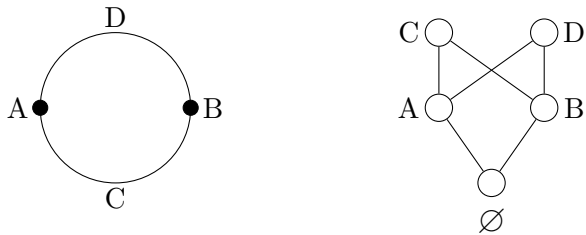

Stanley generalized the notion of a Stanley-Reisner ring to a boolean complex $\Omega$ in \cite{stanley}, as follows. Let $k$ be a field and let $Q$ be the face poset of $\Omega$, including the minimal element $\emptyset$. Let $k[\{z_\alpha\}_{\alpha\in Q}]$ be a polynomial ring with indeterminates indexed by the elements of $Q$. Let $I$ be the ideal of this ring generated by:
\begin{enumerate}
\item the element $z_\emptyset - 1$;
\item all products $z_\alpha z_\beta$ where $\alpha,\beta\in Q$ have no common upper bound; and
\item all elements of the form
\[
z_\alpha z_\beta - z_{\alpha\wedge \beta} \sum_{\gamma\in lub(\alpha,\beta)} z_\gamma
\]
where $\alpha,\beta$ have at least one common upper bound and $lub(\alpha,\beta)$ denotes the (consequently nonempty) set of least upper bounds of $\alpha,\beta$.
\end{enumerate}

The greatest (common) lower bound $\alpha \wedge \beta$ of $\alpha$ and $\beta$ exists and is unique in the above formula because, as remarked above, every lower interval, and in particular the lower interval below any common upper bound for $\alpha,\beta$, is a boolean algebra and therefore a lattice. Thus whenever $\alpha,\beta$ have any common upper bound, they have a unique greatest common lower bound in some lower interval containing them both, and thus in the whole poset.

\begin{definition}\label{def:SRofboolean}
The quotient ring $k[\{z_\alpha\}_{\alpha\in P}]/I$ is called the \textbf{Stanley-Reisner ring} of $\Omega$ and denoted $k[\Omega]$.
\end{definition}

\begin{remark}
Definition \ref{def:SRofboolean} generalizes definition \ref{def:SRofposet}, but in a somewhat subtle way. Given a poset $P$, one can form its chain complex $\Omega$, regarded as a boolean complex, and then the $k[P]$ of \ref{def:SRofposet} will be isomorphic to the $k[\Omega]$ of \ref{def:SRofboolean}; however, the poset $Q$ of the latter definition will not be $P$. Instead, its elements will be {\em chains} in $P$, ordered by inclusion. For example, let $P = B_2\setminus\{\emptyset\}$. Then the elements of $P$ may be abbreviated $1$, $2$, and $12$, and the only incomparable pair consists of $1$ and $2$. Thus 
\[
k[P] = k[y_1,y_2,y_{12}]/(y_1y_2)
\]
according to definition \ref{def:SRofposet}. However, $Q$ consists of the six chains in $P$: the empty chain $\emptyset$, three chains of length 1 ($1$, $2$, and $12$), and two chains of length 2 ($1\subset 12$ and $2 \subset 12$). Thus 
\[
k[\Omega] = k[z_\emptyset, z_1,z_2,z_{12}, z_{1\subset 12}, z_{2\subset 12}] / I
\]
where $I$ is as described above. The isomorphism is given by mapping the $z$ of a given chain to the product of $y$'s corresponding to elements of the chain, for example $z_{1\subset 12}\mapsto y_1y_{12}$. Indeed, the definition of $I$ becomes much more transparent after considering why this map is an isomorphism.
\end{remark}

The ring of interest to us is the invariant ring $k[B_n\setminus\{\emptyset\}]^G$ inside the Stanley-Reisner ring of the poset $B_n\setminus\{\emptyset\}$. This ring can be identified with the Stanley-Reisner ring of a boolean complex using a result of Victor Reiner, as follows. Let $\Delta$ be the order complex of $B_n\setminus\{\emptyset\}$, i.e. the simplicial complex whose vertices are the elements of $B_n\setminus\{\emptyset\}$, and whose faces are the chains in $B_n\setminus\{\emptyset\}$. As a simplicial complex, $\Delta$ is the barycentric subdivision of an $(n-1)$-simplex, thus it is topologically an $(n-1)$-ball. See figure \ref{fig:ordercomplex}.

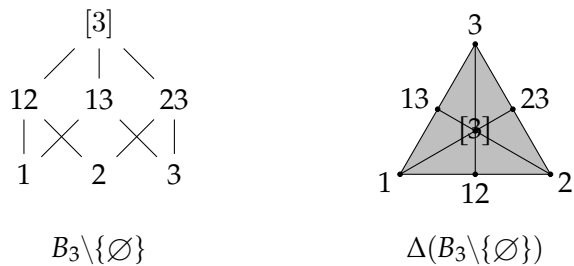
\begin{figure}
\begin{center}
\begin{tikzpicture}

\node (v1) at (-3,1) {$1$};
\node (v2) at (-2,1) {$2$};
\node (v3) at (-1,1) {$3$};
\node (e12) at (-3,2) {$12$};
\node (e13) at (-2,2) {$13$};
\node (e23) at (-1,2) {$23$};
\node (tri) at (-2,3) {[3]};
\node at (-2,0) {$B_3\setminus\{\emptyset\}$};

\foreach \from/\to in {v1/e12, v1/e13, v2/e12, v2/e23, v3/e13, v3/e23, e13/tri, e23/tri, e12/tri} 
    \draw (\from) -- (\to);

\draw [fill = lightgray] (2,1) -- (4,1) -- (3,2.732) -- (2,1);
\draw (2,1) -- (3.5,1.866);
\draw (3,1) -- (3,2.732);
\draw (4,1) -- (2.5,1.866);
\draw [fill] (2,1) circle [radius=0.03];
\draw [fill] (4,1) circle [radius=0.03];
\draw [fill] (3,1) circle [radius=0.03];
\draw [fill] (2.5,1.866) circle [radius=0.03];
\draw [fill] (3,2.732) circle [radius=0.03];
\draw [fill] (3.5,1.866) circle [radius=0.03];
\draw [fill] (3,1.577) circle [radius=0.03];
\node (V1) at (1.8,0.9) {$1$};
\node (V2) at (4.2,0.9) {$2$};
\node (V3) at (3,3) {$3$};
\node (E12) at (3,0.75) {$12$};
\node (E13) at (2.2,2.016) {$13$};
\node (E23) at (3.8,2.016) {$23$};
\node (Tri) at (3,1.577) {$[3]$};
\node at (3,0) {$\Delta(B_3\setminus\{\emptyset\})$};

\end{tikzpicture}
\end{center}
\caption{The poset $B_3\setminus\{\emptyset\}$, and its order complex, which is a 2-ball.}\label{fig:ordercomplex}
\end{figure}

The simplicial complex $\Delta$ carries a natural simplicial action of $S_n$, via the latter's action on $[n]$. The quotient cell complex $\Delta / G$ is usually not simplicial, but it is a boolean complex. This is because $\Delta$ is a {\em balanced} complex, and the action of $G$ is a balanced action. 

\begin{definition}
A boolean complex of dimension $d$ is \textbf{balanced} if there is a labeling of its vertices by $d+1$ labels such that the vertices of any one face have distinct labels. Given such a labeling, a cellular action by a group is a \textbf{balanced action} if it preserves the labeling.
\end{definition}

In the present case, the vertices of $\Delta$ are the nonempty subsets of $[n]$, and thus $\Delta$ is balanced by associating a subset to its cardinality. (Here, $d=n-1$, so the $n$ possible cardinalities give the right number of labels.) The action of $S_n$ is clearly balanced with respect to this labeling. See figure \ref{fig:balanced}.

\begin{figure}
\begin{center}
\begin{tikzpicture}

\draw [fill = lightgray] (2,1) -- (4,1) -- (3,2.732) -- (2,1);
\draw (2,1) -- (3.5,1.866);
\draw (3,1) -- (3,2.732);
\draw (4,1) -- (2.5,1.866);
\draw [fill = purple] (2,1) circle [radius=0.05];
\draw [fill = purple] (4,1) circle [radius=0.05];
\draw [fill = teal] (3,1) circle [radius=0.05];
\draw [fill = teal] (2.5,1.866) circle [radius=0.05];
\draw [fill = purple] (3,2.732) circle [radius=0.05];
\draw [fill = teal] (3.5,1.866) circle [radius=0.05];
\draw [fill = blue] (3,1.577) circle [radius=0.05];
\node [purple] (V1) at (1.8,0.9) {$1$};
\node [purple] (V2) at (4.2,0.9) {$1$};
\node [purple] (V3) at (3,3) {$1$};
\node [teal] (E12) at (3,0.75) {$2$};
\node [teal] (E13) at (2.2,2.016) {$2$};
\node [teal] (E23) at (3.8,2.016) {$2$};
\node [blue] (Tri) at (3.15,1.577) {$3$};

\end{tikzpicture}
\end{center}
\caption{The labeling of the order complex of $B_3\setminus\{\emptyset\}$, showing it is balanced.}\label{fig:balanced}
\end{figure}
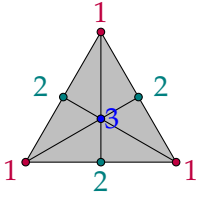

It is straightforward to check that the quotient of a balanced boolean complex by a balanced action is again a balanced boolean complex. (Details are given in \cite{blumsmith}, Lemma 2.5.86.) Thus $\Delta / G$ is a balanced boolean complex.

In \cite[Theorem 2.3.1]{reiner92}, Victor Reiner showed that if a group $G$ acts cellularly and balancedly on a balanced boolean complex $\Omega$, then the invariant ring $k[\Omega]^G$ inside the Stanley-Reisner ring of $\Omega$ is isomorphic to $k[\Omega / G]$, the Stanley-Reisner ring of the quotient boolean complex $\Omega / G$. In the present situation, this gives us 
\begin{equation}\label{eq:invarisSR}
k[\Delta / G] \cong k[B_n\setminus\{\emptyset\}]^G.
\end{equation}
Thus the problem is reduced to showing that $k[\Delta / G]$ is Cohen-Macaulay.


Finally, the Cohen-Macaulayness of $k[\Delta / G]$ can be assessed topologically. In general, the Cohen-Macaulayness of the Stanley-Reisner ring of a boolean complex $\Omega$ is equivalent (just as for a simplicial complex) to a condition on $|\Omega|$, the underlying topological space of $\Omega$, that depends only on its homeomorphism class. Namely, $k[\Omega]$ is Cohen-Macaulay if and only if 
\begin{equation}\label{eq:homvanishing}
\tilde H_i(|\Omega|; k) = 0\text{ and } H_i(|\Omega|,|\Omega|-q;k) = 0
\end{equation}
for all points $q\in |\Omega|$ and all $i<\dim \Omega$. (Here, $\tilde H_i(|\Omega|;k)$ is reduced singular homology and $H_i(|\Omega|,|\Omega| - q;k)$ is relative singular homology.) This theorem is the product of work of Gerald Reisner (building on work of Melvin Hochster), James Munkres, Richard Stanlely, and Art Duval. Reisner proved in \cite{reisner} that for a {\em simplicial} complex $\Omega$, Cohen-Macaulay\-ness of $k[\Omega]$ is equivalent to a homological vanishing condition that a priori depends on the simplicial structure and not just the underlying topological space. Munkres in \cite{munkres} showed that Reisner's condition is equivalent to the purely topological condition stated above. Richard Stanley in \cite{stanley} showed that the direction 
\[
\eqref{eq:homvanishing}\text{ is satisfied for all } q\in|\Omega|\text{ and }i<\dim \Omega\; \Rightarrow \;k[\Omega]\text{ is Cohen-Macaulay}
\] generalizes to boolean complexes, and Art Duval in \cite{duval} showed that this generalization is bidirectional. See \cite[\S 2.5.2]{blumsmith} for more details.

\begin{remark}
Since we only use Stanley-Reisner theory to show the ``if" direction of theorem \ref{mainresult} and thus we only need it to deduce Cohen-Macaulayness, and not the {\em failure} of Cohen-Macaulayness, the proof of \ref{mainresult} only uses Stanley's and not Duval's part of the generalization of \eqref{eq:homvanishing} to boolean complexes.
\end{remark}


Combining the results quoted above, we see that to demonstrate the Cohen-Macaulay\-ness of the ring $k[\mathbf{x}]^G$, it is sufficient to prove that the boolean complex $\Omega = \Delta / G$ satisfies the homological vanishing condition \eqref{eq:homvanishing} for all $x\in |\Delta / G|$ and all $i<n-1$. The proof of the ``if" direction of theorem \ref{mainresult} will consist in showing that this condition holds when $G$ is generated by transpositions, double transpositions, and 3-cycles.

This will be accomplished by quoting a recent result of Christian Lange (see section \ref{sec:if}) that is stated in the language of piecewise-linear (PL) topology, so we also need to recall a few definitions and a basic fact from this field. We follow \cite[Section 3.1]{lange2} and \cite[Chapters 1 and 2]{rourkesanderson} for these details. A \textbf{polyhedron} is a subset $X$ of $\RR^m$ in which each point has a compact cone neighborhood, i.e. given $x\in X$, there is a compact set $K\subset X$ such that (i) the union $S$ of line segments from $x$ to points of $K$ is contained in $X$, (ii) each point of $S\setminus \{x\}$ is on a unique such line segment from $x$, and (iii) $S$ is a neighborhood of $x$ in $X$, i.e. it contains an open subset of $X$ containing $x$. The set $S$ is called a \textbf{star} of $x$ in $X$, and $K$ is called a \textbf{link} of $x$. See figure \ref{fig:linkandstar}.

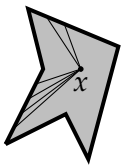
\begin{figure}
\begin{center}
\begin{tikzpicture}

\draw [line width=1.5, fill = lightgray] (-1,-1) -- (-0.5,0) -- (-0.7,0.8) -- (0.5,0.3) -- (0.1,-1.2) -- (-0.2,-0.6) -- (-1,-1);

\draw [fill] (0,0) circle [radius=0.03];
\node (point) at (0,-0.2) {$x$};

\draw (0,0) -- (-0.9,-0.8);
\draw (0,0) -- (-0.8,-0.6);
\draw (0,0) -- (-0.7,-0.4);
\draw (0,0) -- (-0.6,-0.2);

\draw (0,0) -- (-.46, 0.7);
\draw (0,0) -- (-.22, 0.6);

\end{tikzpicture}
\end{center}
\caption{A compact cone neighborhood of a point in $\RR^2$. The link $K$ is drawn in bold, and the star $S$ is the entire set, the union of segments from $x$ to the points of $K$. Some of these segments are also drawn. Note each point of $S\setminus\{x\}$ is on exactly one such segment.}
\label{fig:linkandstar}
\end{figure}

\begin{remark}\label{rmk:polyhedron}
This definition of {\em polyhedron} is a technical device, used here to define the concepts {\em piecewise-linear} and {\em polyhedral star}. It includes the more conventional meaning of a three-dimensional polytope as a special case, but is much, much broader. For example, any open subset of $\RR^n$, or of any polytope, is a polyhedron.

More broadly, our use of PL topology in this paper is only to serve a technical need linking Lange's result to our setting. 
\end{remark}


If $X\subset \RR^m$ and $Y\subset \RR^n$ are polyhedra, a continuous map $f:X\rightarrow Y$ is a \textbf{piecewise-linear} (or \textbf{PL}) \textbf{map} if its graph $\{(x,f(x)):x\in X\}\subset \RR^{m+n}$ is a polyhedron. A \textbf{piecewise-linear} (or \textbf{PL}) \textbf{space} is a second-countable, Hausdorff topological space equipped with a covering by open sets $U_i$, each with a homeomorphism $\varphi_i:X_i\rightarrow U_i$ from a polyhedron $X_i$ in some $\RR^{m_i}$, such that the transition maps 
\[
\varphi_j^{-1}\circ\varphi_i|_{\varphi_i^{-1}(U_i\cap U_j)}
\]
are PL. A PL space is a \textbf{PL manifold} (with or without boundary) if the charts $X_i$ can be taken to be open subsets of $\RR^n$ or the half-space $\RR^{n-1}\times \RR^{\geq 0}$.

A subset $P$ of a PL space $Y$ is called a polyhedron if for each of the charts $\varphi_i:X_i\rightarrow U_i\subset Y$, the preimage $\varphi_i^{-1}(P)\subset X_i\subset \RR^{m_i}$ is a polyhedron.

If $X\subset \RR^n$ is a polyhedron and $x\in X$, one may always find a link and star for $x$ that are polyhedra (\cite[p. 5]{rourkesanderson}). It then follows from the definitions that if $Y$ is a PL space, any point $y$ of $Y$ has a neighborhood $S$ contained in some $U_i\ni y$, such that the preimage $\varphi_i^{-1}(S)\subset X_i$ is both a polyhedron and a star of $\varphi_i^{-1}(y)$ in $X_i$. We will refer to such an $S$ as a \textbf{polyhedral star} of $y$.

The key fact we need is that if $X$ is a polyhedron and $x\in X$, then any two polyhedral stars of $x$ in $X$ are PL-homeomorphic, in other words the star is a PL-homeomorphism invariant of $x$ (\cite[pp.~20--21]{rourkesanderson}). It follows from the above discussion that the same is true in any PL space.

If $Y$ is a PL manifold, one may take each chart $X_i$ to be an open subset in $\RR^n$ or $\RR^{n-1}\times \RR^{\geq 0}$. In any open subset of $\RR^n$, the star of a point $(x_1,\dots,x_n)$ may be taken to be the cube $[x_1-\varepsilon,x_1+\varepsilon]\times\dots\times[x_n-\varepsilon, x_n+\varepsilon]$ for sufficiently small $\varepsilon > 0$; and in $\RR^{n-1}\times \RR^{\geq 0}$ it can be taken to be the intersection of this cube with the closed half-space $\{x_n\geq 0\}$. In all cases, this is topologically a closed ball. It then follows from the fact quoted in the previous paragraph that {\em every} polyhedral star in a PL manifold is topologically a ball.

The ``if" direction of theorem \ref{mainresult} will be proven by quoting the result of Lange mentioned above to show that if $G$ is generated by transpositions, double transpositions, and 3-cycles, then $\Delta/G$ is a polyhedral star of a point in a PL manifold, and therefore a ball. Thus it meets the homological vanishing criterion described above, regardless of the field $k$.

\subsection{Generalities about group actions on a ring}\label{sec:groupactions}

The purpose of this section is to develop the commutative algebra needed to prove the general results in \S\ref{sec:inertiaCM}, which are then used in section \S\ref{sec:onlyif} to prove the ``only-if" direction of theorem \ref{mainresult}.

Let $1$ denote the group identity. (In commutative diagrams, let it also denote a trivial group.) Let $A^G$ denote the ring of invariants, and similarly for any subgroup of $G$. It is well known that $A$ is always integral over $A^G$ (\cite[Chapitre V \S 1.9, Proposition 22]{bourbaki}).

Let $\mfP\triangleleft A$ be a prime ideal.

Recall that the \textbf{decomposition group} $D_G(\mfP)$ of $\mfP$ is the stabilizer of $\mfP$ in $G$:
\[
D_G(\mfP) = \{g\in G: g\mfP = \mfP\}.
\]
The decomposition group acts on the integral domain $A/\mfP$. The \textbf{inertia group} $I_G(\mfP)$ of $\mfP$ is the kernel of this action: 
\[
I_G(\mfP) = \left\{g\in G:(g-1)A\subset \mfP\right\},
\]
where
\[
(g-1)A = \{ga - a: a\in A\}.
\]

The notations $I_G(\mfP)$ and $D_G(\mfP)$ implicitly specify the ring $A$ being acted on by $G$, since $\mfP$ belongs to $A$.

We recall some basic facts in this setup (\cite[Chapitre V \S 2.2, Th\'eor\`eme 2]{bourbaki}), which we use freely in what follows: (i) $G$ acts transitively on the prime ideals of $A$ lying over $\mfP^\star = \mfP\cap A^G$; and (ii) the extension of residue fields $\kappa(\mfP)/\kappa(\mfP^\star)$ is a normal field extension, and the canonical map from $D_G(\mfP)$ to the group of $\kappa(\mfP^\star)$-automorphisms of $\kappa(\mfP)$ is a surjection with kernel $I_G(\mfP)$, i.e. the sequence
\[
\begin{CD}
1@>>> I_G(\mfP) @>>> D_G(\mfP) @>>> \Aut_{\kappa(\mfP^\star)} (\kappa(\mfP)) @>>> 1
\end{CD}
\]
is exact.

If $N\triangleleft G$ is a normal subgroup, then the quotient group $G/N$ acts on the invariant ring $A^N$, and the decomposition and inertia groups in $G$ and $G/N$ relate straightforwardly. Note that, by their definitions, $I_N(\mfP) = I_G(\mfP)\cap N$ and $D_N(\mfP) = D_G(\mfP)\cap N$.

\begin{lemma}\label{lem:GNinertia}
We have
\[
D_{G/N}(\mfP\cap A^N) \cong D_G(\mfP)/D_N(\mfP)
\]
and
\[
I_{G/N}(\mfP\cap A^N) \cong I_G(\mfP)/I_N(\mfP).
\]
\end{lemma}

We believe this and the next lemma may be well-known; however, as we were unable to locate references, we include full proofs.

\begin{proof}
The sequences
\[
\begin{CD}
1 @>>> D_N(\mfP) @>>> D_G(\mfP) @>\varphi>> D_{G/N}(\mfP\cap A^N) @>>> 1
\end{CD}
\]
and
\[
\begin{CD}
1 @>>> I_N(\mfP) @>>> I_G(\mfP) @>\psi>> I_{G/N}(\mfP\cap A^N) @>>> 1
\end{CD}
\]
are exact in the first and second positions by the definitions; we have to prove surjectivity of $\varphi$ and $\psi$.

Consider $\varphi$ first. Suppose $g\in G$ is such that its image $\overline{g}$ in $G/N$ lies in $D_{G/N}(\mfP\cap A^N)$. Then, setting $\mfQ = g\mfP$, we have 
\[
\mfQ \cap A^N = \mfP\cap A^N.
\]
All primes of $A$ that intersect $A^N$ in $\mfP\cap A^N$ lie in the same orbit of $N$. Thus there exists $n\in N$ with $n\mfQ = \mfP$. Therefore $ng\mfP = \mfP$, i.e. $ng\in D_G(\mfP)$, and we have $\varphi(ng) = \overline{g}$. So $\varphi$ is surjective.

We establish the surjectivity of $\psi$ with a diagram chase. Let $\mfP' = \mfP \cap A^N$ and let $\mfP^\star = \mfP\cap A^G$. We have the following commutative diagram:
\[
\begin{CD}
  @. 1 @. 1 @. 1 @.  \\
 @.  @VVV  @VVV  @VVV  @.\\
1 @>>> I_N(\mfP) @>>> I_G(\mfP) @>\psi>> I_{G/N}(\mfP') @. \\
@.  @VVV @VVi_GV @VVi_{G/N}V @. \\
1 @>>> D_N(\mfP) @>\imath_D>> D_G(\mfP) @>\varphi>> D_{G/N}(\mfP') @>>> 1 \\
@.  @Vp_NVV @VVp_GV @VVp_{G/N}V @. \\
1 @>>> \Aut_{\kappa(\mfP')}(\kappa(\mfP)) @>>\imath_\kappa> \Aut_{\kappa(\mfP^\star)}(\kappa(\mfP)) @>>\xi> \Aut_{\kappa(\mfP^\star)}(\kappa(\mfP')) @>>> 1\\
@. @VVV @VVV @VVV @.\\
 @. 1 @. 1 @. 1 @.  
\end{CD}
\]
where $\kappa(\mfP),\kappa(\mfP'),\kappa(\mfP^\star)$ are the residue fields. The first and second row are exact by what we have just done. The third row is exact by consideration of the definitions and the fact that $\kappa(\mfP)$ is normal over $\kappa(\mfP')$ (by \cite[Chapitre V \S 2.2, Th\'eor\`eme 2(ii)]{bourbaki}, as recalled above), since field automorphisms always extend to normal extensions. The columns are also exact by \cite[Chapitre V \S 2.2, Th\'eor\`eme 2(ii)]{bourbaki}.

Let $g\in I_{G/N}(\mfP')$ be arbitrary and consider $i_{G/N}(g)$. Since $\varphi$ is surjective, there is a $y\in D_G(\mfP)$ with $\varphi(y) = i_{G/N}(g)$. Then
\[
1 = p_{G/N}\circ i_{G/N}(g) = p_{G/N}\circ \varphi(y) = \xi\circ p_G(y),
\]
so that $p_G(y)\in \ker \xi = \im \imath_\kappa$. Thus there is a $z\in \Aut_{\kappa(\mfP')}(\kappa(\mfP))$ with $\imath_\kappa(z) = p_G(y)$. Since $p_N$ is surjective, we have a $z'\in D_N(\mfP)$ with $p_N(z') = z$. Now consider
\[
y^\star = \imath_D(z')^{-1}y\in D_G(\mfP).
\]
We have
\begin{align*}
p_G(y^\star) &= p_G\circ \imath_D(z')^{-1} p_G(y) \\
&= \imath_\kappa\circ p_N(z')^{-1}p_G(y)\\
&= \imath_\kappa(z)^{-1}p_G(y)\\
&= p_G(y)^{-1}p_G(y)\\
&= 1.
\end{align*}
Thus $y^\star\in \ker p_G = \im i_G$, so there exists $g'\in I_G(\mfP)$ with $i_G(g') = y^\star$. Then
\begin{align*}
i_{G/N}\circ\psi (g') &= \varphi\circ i_G(g')\\
&= \varphi(y^\star)\\
&= \varphi\left(\imath_D(z')^{-1}y\right)\\
&= \varphi\circ\imath_D(z')^{-1}\varphi(y)\\
&= 1^{-1}i_{G/N}(g)=i_{G/N}(g).
\end{align*}
Since $i_{G/N}$ is injective, we can conclude $\psi(g')=g$. Thus $\psi$ is surjective.
\end{proof}




The inertia group of a prime that survives a base change remains stable under that base change, and the decomposition group can only shrink:

\begin{lemma}\label{lem:basechangepreservesinertia}
Let $C$ be an arbitrary $A^G$-algebra, and let $B:= A\otimes_{A^G}C$. Let $G$ act on $B$ through its action on $A$ and trivial action on $C$. If there is a prime $\mfQ$ of $B$ pulling back to $\mfP$ in $A$, then $D_G(\mfQ)\subset D_G(\mfP)$, and $I_G(\mfQ) = I_G(\mfP)$.
\end{lemma}

\begin{proof}
Let $\tau: A\rightarrow B$ be the canonical map. By construction, $\tau$ is $G$-equivariant. Thus if $g\in G$ stabilizes $\mfQ\triangleleft B$ setwise, it also stabilizes the preimage $\mfP\triangleleft A$ setwise, and it follows that $D_G(\mfQ)\subset D_G(\mfP)$.

When $g\in D_G(\mfQ)$ and therefore $\in D_G(\mfP)$, it has an induced action on both $B/\mfQ$ and $A/\mfP$, and the $G$-equivariance of $\tau$ then implies that the induced map
\[
\overline\tau: A/\mfP \rightarrow B/\mfQ
\]
is $\langle g\rangle$-equivariant. If also $g\in I_G(\mfQ)$, then its action on $B/\mfQ$ is trivial. Since $\mfP$ is the full preimage of $\mfQ$, $\overline\tau$ is an injective map, and it follows that $g$'s action on $A/\mfP$ is also trivial, i.e. $g\in I_G(\mfP)$. Thus $I_G(\mfQ)\subset I_G(\mfP)$.

In the other direction, suppose $g\in I_G(\mfP)$. By \cite[Chapter 1, Corollary 1.13]{liu}, we have a canonical isomorphism
\begin{equation}\label{liuiso}
B/\tau(\mfP) B \cong A/\mfP\otimes_{A^G} C.
\end{equation}
Using only the fact that $g\in D_G(\mfP)$ and the $G$-equivariance of $\tau$, we already know that $g$ fixes $\mfP$ and $\tau(\mfP)$ setwise, and thus has well-defined actions on $A/\mfP$ and $B/\tau(\mfP)B$ that coincide via \eqref{liuiso}. But because $g$ is actually in $I_G(\mfP)$, the action on $A/\mfP$ is trivial, and therefore, by \eqref{liuiso}, the action of $g$ on $B/\tau(\mfP) B$ is also trivial. 

In other words, $g$ fixes the cosets of the additive subgroup $\tau(\mfP)B$ of $B$ setwise. Since $\mfQ$ pulls back to $\mfP$, it contains the image of $\mfP$, thus we have $\mfQ\supset \tau(\mfP)B$. Then the cosets of $\mfQ$ are unions of cosets of $\tau(\mfP)B$, and therefore $g$ fixes these setwise as well. In other words, $g$ acts trivially on $B/\mfQ$, i.e. $g\in I_G(\mfQ)$. Thus $I_G(\mfP)\subset I_G(\mfQ)$, and we conclude $I_G(\mfP) = I_G(\mfQ)$.
\end{proof}

\begin{remark}
Examining the proof of lemma \ref{lem:basechangepreservesinertia}, we see why the analogous equality to $I_G(\mfP)=I_G(\mfQ)$ may fail for decomposition groups. If $g\in D_G(\mfP)$, then we do have the $\langle g\rangle$-equivariant isomorphism \eqref{liuiso}, and therefore $g$ does act on the cosets of $\tau(\mfP)B$ in $B$, but the only one we know it fixes is $\tau(\mfP)B$ itself. In particular, $\mfQ$, which may be the union of many of these cosets, need not be fixed setwise, so that $g\notin D_G(\mfQ)$.
\end{remark}

Henceforth, let $\mfp$ be a prime of $A^G$. Our goal is to show that, in a suitable sense, the local structure of $A^G$ at $\mfp$ is determined by the inertia group of a prime of $A$ lying over $\mfp$. The precise statement is lemma \ref{ringslice} below. It is stated by Michel Raynaud in \cite[Chapitre X \S 1, Corollaire 1]{raynaud}, with lines of proof indicated. Because it is central to our results, we develop in detail the notation and tools that will be required to state and prove this lemma.

Let $C_\mathfrak{p}^{hs}$ be the strict henselization (see \cite[Chapitre VIII, Definition 4]{raynaud} or \cite[Definition 18.8.7]{grothendieck}) of $A^G$ at $\mathfrak{p}$, with respect to some embedding of $\kappa(\mfp)$ in its separable closure. Then $C_\mathfrak{p}^{hs}$ is faithfully flat over $(A^G)_{\mathfrak{p}}$, and of relative dimension zero (\cite[Proposition 18.8.8(iii)]{grothendieck}). Furthermore, $C_\mfp^{hs}$ and $(A^G)_\mfp$ are simultaneously noetherian (\cite[Proposition 18.8.8(iv)]{grothendieck}), and
\[
A^{hs}_{\mathfrak{p}}:=A \otimes_{A^G} C_\mathfrak{p}^{hs}
\]
is integral over $C_\mathfrak{p}^{hs}$ (as it is a base change of the integral morphism $A^G\rightarrow A$). 
Moreover, $G$ acts on $A^{hs}_{\mathfrak{p}}$ via the the first component of the tensor product, so that the map $A\rightarrow A_\mfp^{hs}$ is $G$-equivariant, and 
\[
(A^{hs}_{\mathfrak{p}})^G = C_\mathfrak{p}^{hs}
\]
since $A^G \rightarrow (A^G)_\mathfrak{p} \rightarrow  C_\mathfrak{p}^{hs}$ is  flat and the functor of invariants commutes with flat base change.


Let $\mfP$ be a prime ideal of $A$ lying over $\mfp$, and let $\mfQ$ be a prime ideal of $A_\mfp^{hs}$ lying over  the maximal ideal of $C_\mfp^{hs}$ corresponding to $\mfp$, and pulling back to $\mfP$ in $A$.

From lemma \ref{lem:basechangepreservesinertia}, we have that
\[
I_G(\mathfrak{Q})= I_G(\mathfrak{P}).
\]

The action of $G$ on $A^{hs}_\mfp$ induces an action on its ideals. Since $A^{hs}_{\mathfrak{p}}$ is integral over $C^{hs}_{\mathfrak{p}}$, all of its maximal ideals lie over the one maximal of $C^{hs}_\mfp$. Because $C^{hs}_\mfp$ is the invariant ring under the action of $G$, this implies (\cite[Chapitre V \S 2.2, Th\'eor\`eme 2(i)]{bourbaki}) that the maximal ideals of $A^{hs}_\mfp$ comprise a single orbit for the action on ideals. The maximals are therefore finite in number. We denote them by $\mfM_1(=\mfQ),\dots,\mfM_s$. 



The product of canonical localization homomorphisms
\begin{equation} \label{iso} 
\phi : A_\mfp^{hs}\rightarrow \prod_{j=1}^s (A_\mfp^{hs})_{\mfM_j}
\end{equation}
is an isomorphism. Indeed, $A_\mfp^{hs}$ is the inductive limit of $C_\mfp^{hs}$-finite subalgebras (since it is integral over $C_\mfp^{hs}$). Since $A_\mfp^{hs}$ has only $s$ maximals, there exists a finite subalgebra containing $s$ maximals. Now view $A_\mfp^{hs}$ as the inductive limit just of the finite subalgebras that contain this one. For each of them, the analogous product of canonical localization morphisms is an isomorphism because $C_\mfp^{hs}$ is henselian (see \cite[Chapitre I, \S1 D\'{e}finition 1 and Proposition 3]{raynaud}); then the statement about \eqref{iso} follows because inductive limits commute with finite products.

\begin{lemma}\label{lem:noetherian}
If $A$ is a noetherian ring, then $A_\mfp^{hs}$ is noetherian too.
\end{lemma}

\begin{proof}
Because of the isomorphism \eqref{iso}, it suffices to show that the localizations of $A_\mfp^{hs}$ at its maximal ideals $\mfM_j$ are noetherian rings, and because the action of $G$ on $A_\mfp^{hs}$ by automorphisms is transitive on these maximals, it suffices to show this for a single maximal. We will do this by showing that there is a maximal ideal $\mfM_j$ of $A_\mfp^{hs}$ such that
\[
(A_\mfp^{hs})_{\mfM_j}
\]
is isomorphic to the strict henselization of the noetherian local ring $A_\mfP$, whereupon the result will follow because strict henselization preserves noetherianity (\cite[Proposition 18.8.8(iv)]{grothendieck}).

Consider the local ring $(A^G)_\mfp$. By slight abuse of notation, let us call its maximal ideal $\mfp$. Note that the residue field $\kappa(\mfp)$ is the same whether $\mfp$ refers to the prime in $A^G$ or in $(A^G)_\mfp$, so we can write $\kappa(\mfp)$ without ambiguity.  Then the maximal ideals in the ring
\[
B := A\otimes_{A^G} (A^G)_\mfp
\]
are in bijection with the prime ideals of $A$ lying over $\mfp\triangleleft A^G$. There are finitely many of these since they are subject to a transitive action by $G$, so $B$ is semilocal. It is also integral as an extension of $(A^G)_\mfp$ since this is a base change of the integral extension $A^G\subset A$. One of the prime ideals over $\mfp$ in $A$ is $\mfP$. By the same abuse of notation, let $\mfP$ also refer to the corresponding ideal in $B$; again, this does not introduce ambiguity when writing $\kappa(\mfP)$. Note that $B_\mfP = A_\mfP$ because $B$ is obtained from $A$ by inverting some but not all of the elements in the complement of $\mfP$.

Because $B$ is semilocal and integral over $(A^G)_\mfp$ (and $\mfP$ and $\mfp$ are maximal ideals of these rings respectively), if we can show that the extension of residue fields $\kappa(\mfP) / \kappa(\mfp)$ has finite separable degree, then it will follow from \cite[Proposition 18.8.10 and its proof, and Remarque 18.8.11]{grothendieck} that the strict henselization
\[
(B_{\mfP})^{hs}
\]
of the localization $B_{\mfP}$ (with respect to some embedding of its residue field in a separable closure) is isomorphic to the localization of
\[
B\otimes_{(A^G)_\mfp} C_\mfp^{hs}
\]
at some maximal ideal, since $C_\mfp^{hs}$ is a strict henselization of $(A^G)_\mfp$. But we also have
\begin{align*}
B\otimes_{(A^G)_\mfp} C_\mfp^{hs} &= A\otimes_{A^G} (A^G)_\mfp \otimes_{(A^G)_\mfp} C_\mfp^{hs} \\
&= A\otimes_{A^G} C_\mfp^{hs}\\
&= A_\mfp^{hs}.
\end{align*}
Thus the conclusion from \cite[18.8.10 and 18.8.11]{grothendieck} will actually be that
\[
(A_\mfP)^{hs} = (B_\mfP)^{hs} \cong (A_\mfp^{hs})_{\mfM_j}
\]
for some maximal ideal $\mfM_j$ of $A_\mfp^{hs}$. This is the desired conclusion, so it remains to show that $\kappa(\mfP)/\kappa(\mfp)$ has finite separable degree.

Now return $\mfp,\mfP$ to the setting of $A^G$ and $A$, recalling that the residue fields $\kappa(\mfp),\kappa(\mfP)$ do not change. From \cite[Chapitre V, \S 2.2(ii)]{bourbaki} we have that $\kappa(\mfP)/\kappa(\mfp)$ is a normal field extension, and the group of $\kappa(\mfp)$-automorphisms of $\kappa(\mfP)$ is isomorphic to
\[
D_G(\mfP)/I_G(\mfP).
\]
This is a subquotient of the finite group $G$ and is therefore finite. For a normal field extension, infinite separable degree would imply infinitely many automorphisms. Thus $\kappa(\mfP)/\kappa(\mfp)$ is an extension of finite separable degree, and the proof is complete.
\end{proof}

The action of $G$ on $A_\mfp^{hs}$ induces, via the isomorphism $\phi$ of \eqref{iso}, an action on $\prod_1^s(A_\mfp^{hs})_{\mfM_j}$: it is the unique action on this ring such that $\phi$ is $G$-equivariant. Because $\phi$ is the product of the canonical localization maps
\[
\phi_j:A^{hs}_\mfp \rightarrow(A_\mfp^{hs})_{\mfM_j},
\]
it is possible to write down this action explicitly. Via the isomorphism $\phi$ of \eqref{iso} we associate uniquely to $a\in A_\mfp^{hs}$ the $s$-tuple
\begin{equation}  \label{coord}
\phi(a) = (a_{\mfM_1}, \cdots ,a_{\mfM_s})\in \prod_{j=1}^s (A^{hs}_{\mathfrak{p}})_{\mathfrak{M}_j}
\end{equation} 
where each $a_{\mfM_j}$ is the image in $(A_\mfp^{hs})_{\mfM_j}$ of $a$ under $\phi_j$. If $g\in G$ maps $\mfM_i$ to $\mfM_j$, then it also induces an isomorphism
\begin{align*}
(A_\mfp^{hs})_{\mfM_i} &\xrightarrow{g} (A_\mfp^{hs})_{\mfM_j}\\
a/s &\mapsto ga/gs
\end{align*}
of the localizations that makes the following square
\[
\begin{CD}
A_\mfp^{hs} @>g>> A_\mfp^{hs} \\
@V\phi_iVV @VV\phi_jV\\
(A_\mfp^{hs})_{\mfM_i} @>>g> (A_\mfp^{hs})_{\mfM_j}
\end{CD}
\]
commutative. By such isomorphisms, $G$ acts on the disjoint union of the localizations $(A_\mfp^{hs})_{\mfM_j}$. Given an $\alpha\in (A_\mfp^{hs})_{\mfM_i}$, if one chooses $a\in A_\mfp^{hs}$ with $\phi_i(a) = \alpha$, then the commutativity of this square can be rewritten as 
\[
g\alpha = \phi_j(ga).
\]
Note that this statement is true regardless of the choice of $a$. For any such choice, writing $\alpha = a_{\mfM_i}$ and $\phi_j(ga) = (ga)_{\mfM_j} = (ga)_{g(\mfM_i)}$, this becomes
\[
g(a_{\mfM_i}) = (ga)_{g(\mfM_i)},
\]
or equivalently,
\begin{equation}\label{actloc}
g(a_{g^{-1}(\mfM_j)}) = (ga)_{\mfM_j}.
\end{equation}
Thus, for any $a\in A_\mfp^{hs}$, the $i$th coordinate of $\phi(a)$ determines the $j$th coordinate of $\phi(ga)$, without requiring additional information about $a$. Then the action of $G$ on $\prod_1^s (A_\mfp^{hs})_{\mfM_j}$ induced by $\phi$ may be written 
\begin{equation}\label{actprod}
g(a_{\mfM_1},\dots,a_{\mfM_s}) = \left(g\left(a_{g^{-1}(\mfM_1)}\right),\dots, g\left(a_{g^{-1}(\mfM_s)}\right)\right).
\end{equation}
Indeed, if $a\in A_\mfp^{hs}$, then the left side of this formula is $g\phi(a)$, and the right side is $\phi (ga)$ by \eqref{actloc}.

Because $I_G(\mfQ)$ stabilizes $\mfQ = \mfM_1$, it acts on $(A_\mfp^{hs})_\mfQ$. In this setting, we have the following lemma. As mentioned above, this lemma was stated by Michel Raynaud in \cite[Chapitre X \S 1, Corollaire 1]{raynaud}, with the proof sketched. It is the needed statement that the local structure of $A^G$ is determined by the inertia groups. Because it is critical to our results, we give a detailed proof.

\begin{lemma}[Raynaud]\label{ringslice}
We have a ring isomorphism 
\[
(A^{hs}_\mfp)_\mfQ^{I_G(\mfP)}\cong C_\mathfrak{p}^{hs}.
\]
\end{lemma}

\begin{proof}
Recall that $I_G(\mfP) = I_G(\mfQ)$.
Let $g_1, \cdots , g_s  \in G$ be a set of left coset representatives for $G / I_G(\mfP)$, with $g_1$ the identity.  
Since $C^{hs}_{\mathfrak{p}}$ is strictly henselian, its residue field is separably closed, so there are no nontrivial automorphisms of $\kappa(\mfQ)$ over it. Since the group of automorphisms of $\kappa(\mfQ)/\kappa(C_\mfp^{hs})$ is isomorphic to $D_G(\mfQ)/I_G(\mfQ)$, we have $D_G(\mfQ) = I_G(\mfQ)$, so that $I_G(\mfQ)$, which equals $I_G(\mfP)$, is the stabilizer of $\mfQ$. Thus, if we put $\mathfrak{M}_j:=g_j\mathfrak{Q}$, then the ideals $\mathfrak{M}_1$, ..., $\mathfrak{M}_s$ are exactly the maximal ideals of $A^{hs}_{\mathfrak{p}}$, and all of the above discussion applies.

We claim that if one restricts the canonical localization map 
\[
\phi_1:A_\mfp^{hs}\rightarrow (A_\mfp^{hs})_\mfQ
\]
to $C_\mfp^{hs}$, one obtains an isomorphism onto $(A_\mfp^{hs})_\mfQ^{I_G(\mfQ)}$. We see this as follows:

The map $\phi_1$ is the composition of $\phi$ with projection to the first coordinate. Because \eqref{actprod} makes $\phi$ a $G$-equivariant isomorphism, $a\in A_\mfp^{hs}$ is in $C_\mfp^{hs} = (A_\mfp^{hs})^G$ if and only if
\begin{equation}\label{invariantif}
\left(g\left(a_{g^{-1}(\mfM_1)}\right),\dots, g\left(a_{g^{-1}(\mfM_s)}\right)\right) = (a_{\mfM_1},\dots,a_{\mfM_s})
\end{equation}
for all $g\in G$. From \eqref{invariantif}, we will deduce the following:
\begin{enumerate}[label=(\alph*)]
\item If $a\in C_\mfp^{hs}$ is an arbitrary $G$-invariant, then $\phi_1(a)$ is invariant under $I_G(\mfP)$. Thus $\phi_1(C_\mfp^{hs})$ is contained in $(A_\mfp^{hs})_\mfQ^{I_G(\mfP)}$.\label{phi1imagecontained}
\item If $a\in C_\mfp^{hs}$ is an arbitrary $G$-invariant, then all the coordinates of $\phi(a)$ are determined by the first coordinate. Thus $a$ itself is determined by $\phi_1(a)$. In other words, the restriction of $\phi_1$ to $C_\mfp^{hs}$ is injective.\label{phi1inj}
\item If $\alpha\in (A_\mfp^{hs})_\mfQ^{I_G(\mfP)}$ is arbitrary, there exists an $a\in C_\mfp^{hs}$ with $\phi_1(a) = \alpha$. Thus the restriction of $\phi_1$ to $C_\mfp^{hs}$ is surjective.\label{phi1surj}
\end{enumerate}
This will suffice to establish the lemma.

To prove \ref{phi1imagecontained}, take $g\in I_G(\mfP)$. The condition in the first coordinate of \eqref{invariantif} is
\[
g(a_{g^{-1}(\mfM_1)}) = a_{\mfM_1}.
\]
For $g\in I_G(\mfP)=D_G(\mfQ)$, we have $g^{-1}(\mfM_1) = \mfM_1=\mfQ$, and this condition becomes 
\[
g(a_\mfQ) = a_\mfQ.
\]
Thus for the $G$-invariant $a$, we have that $a_\mfQ=\phi_1(a)$ is an $I_G(\mfP)$-invariant. Therefore, $\phi_1(C_\mfp^{hs})$ is contained in $(A_\mfp^{hs})_\mfQ^{I_G(\mfP)}$.

For \ref{phi1inj}, consider $g=g_j$ for $j=1,\dots,s$. The condition in the $j$th coordinate of \eqref{invariantif} is
\[
g(a_{g^{-1}(\mfM_j)}) = a_{\mfM_j}.
\]
Since $g_j^{-1}(\mfM_j) = \mfQ$, this becomes
\[
g_j(a_\mfQ) = a_{\mfM_j}.
\]
Letting $j=1,\dots,s$, this shows that if $a$ is a $G$-invariant, then all the coordinates of $\phi(a)$ are determined by $a_\mfQ$, which is $\phi_1(a)$, so $a$ itself is determined by $\phi_1(a)$. Therefore, the restriction of $\phi_1$ to $C_\mfp^{hs}$ is injective.

Lastly, for \ref{phi1surj}, let $\alpha \in (A_\mfp^{hs})_\mfQ^{I_G(\mfP)}$ be arbitrary. We construct a specific $a\in A_\mfp^{hs}$ with $\phi_1(a) = \alpha$, and show it lies in $C_\mfp^{hs}$. Set 
\[
a_{\mfM_j} := g_j(\alpha)
\]
for $j=1,\dots,s$, and let
\[
a := \phi^{-1}\left(a_{\mfM_1},\dots,a_{\mfM_s}\right) \in A_\mfp^{hs}.
\]
Note that this $a$ satisfies $\phi_1(a) = a_{\mfM_1} = g_1(\alpha) = \alpha$ since $g_1$ is the identity. To show that it also lies in $C_\mfp^{hs} = (A_\mfp^{hs})^G$, it is necessary and sufficient to show that $\phi(a)$ satisfies \eqref{invariantif} for all $g\in G$, i.e. that
\begin{equation}\label{isinvarj}
g(a_{g^{-1}(\mfM_j)}) = a_{\mfM_j}
\end{equation}
for all $g\in G$ and all $j=1,\dots,s$.

To do this, we first establish that
\begin{equation}\label{aQ}
a_{g(\mfQ)} = g(a_\mfQ)
\end{equation}
for all $g\in G$, and then use this to show \eqref{isinvarj} for all $g$ and all $j$.

To see \eqref{aQ}, first recall that $\alpha = a_{\mfM_1} = a_\mfQ$, and then use this and $\mfM_j = g_j(\mfQ)$ to rewrite the definition of each $a_{\mfM_j}$:
\[
a_{g_j(\mfQ)} = g_j(a_\mfQ).
\]
This establishes \eqref{aQ} in the particular case that $g$ is one of $g_1,\dots,g_s$. An arbitrary $g\in G$ has the form $g_jh$ for some $g_j$ and some $h\in I_G(\mfP)$. Since $\mfQ$ and $a_\mfQ = \alpha$ are both $I_G(\mfP)$-invariant, we have
\[
a_{g(\mfQ)} = a_{g_jh(\mfQ)} = a_{g_j(\mfQ)} = g_j(a_\mfQ) = g_jh(a_\mfQ) = g(a_\mfQ),
\]
and \eqref{aQ} is established for all $g\in G$.

Now we deduce \eqref{isinvarj}. If $g\in G$ is arbitrary, then 
\[
a_{g^{-1}(\mfM_j)} = a_{g^{-1}g_j(\mfQ)}
\]
because $g_j(\mfQ) = \mfM_j$, and 
\[
a_{g^{-1}g_j(\mfQ)} = g^{-1}g_j(a_\mfQ)
\]
by \eqref{aQ}. Thus $a_{g^{-1}(\mfM_j)} = g^{-1}g_j(a_\mfQ)$, and applying $g$ to the left on both sides yields
\[
g(a_{g^{-1}(\mfM_j)}) = g_j(a_\mfQ) = a_{\mfM_j},
\]
so condition \eqref{isinvarj} is met for all $g$ and all $j$, i.e. \eqref{invariantif} is met for all $g$. Thus 
\[
a \in (A_\mfp^{hs})^G = C_\mfp^{hs}.
\]
Since $\alpha\in (A_\mfp^{hs})_\mfQ^{I_G(\mfP)}$ was arbitrary, this shows that the restriction of $\phi_1$ to $C_\mfp^{hs}$ is surjective onto $(A_\mfp^{hs})_\mfQ^{I_G(\mfP)}$, completing the proof of isomorphism.
\end{proof}

\section{Inertia groups and Cohen-Macaulayness of invariant rings}\label{sec:inertiaCM}

Using lemma \ref{ringslice}, we can show that the Cohen-Macaulayness of a ring of invariants at a prime ideal $\mfp$ can always be tested in a faithfully flat neighborhood of $\mathfrak{p}$, and only depends on the action of the inertia group considered around this neighborhood. The precise statement is theorem \ref{m}.

We use this to derive an obstruction to Cohen-Macaulayness for a characteristic $p$ ring that will apply in the situation of theorem \ref{mainresult} to prove the ``only-if" direction. The statement is proposition \ref{prop:pobstruction}.

In all of what follows, we use the notation of \S \ref{sec:groupactions}: $A$ is a commutative, unital ring endowed with a faithful action of a finite group $G$; and if $\mfp$ is a prime ideal of $A^G$, then $C_\mfp^{hs}$ is the strict henselization of $A^G$ at $\mfp$, and $A_\mfp^{hs}$ is 
\[
A\otimes_{A^G}C_\mfp^{hs},
\]
with $G$ acting through its action on $A$ (and trivially on $C_\mfp^{hs}$).

\begin{theorem}\label{m}
Assume that $A^G$ is noetherian. Then the following assertions are equivalent:
\begin{enumerate}
\item $A^G$ is Cohen-Macaulay.
\item For every prime ideal $\mathfrak{p}$ of $A^G$, and for every prime ideal $\mathfrak{Q}$ of $A^{hs}_{\mathfrak{p}}$ lying over $\mfp C_\mathfrak{p}^{hs}$ and pulling back to a prime $\mathfrak{P}$ of $A$ lying over $\mathfrak{p}$, 
\[
{{(A^{hs}_{\mathfrak{p}})_\mathfrak{Q}}}^{I_G(\mathfrak{P})}
\]
is Cohen-Macaulay.
\item For every maximal ideal $\mathfrak{p}$ of $A^G$, there is some prime ideal $\mathfrak{Q}$ of $A^{hs}_{\mathfrak{p}}$ lying over $\mfp C_\mathfrak{p}^{hs}$ and pulling back to a prime $\mathfrak{P}$ of $A$ lying over $\mathfrak{p}$, such that 
\[
{{(A^{hs}_{\mathfrak{p}})_\mathfrak{Q}}}^{I_G(\mathfrak{P})}
\]
is Cohen-Macaulay.
\end{enumerate} 
\end{theorem} 

\begin{proof}
Clearly (2)$\Rightarrow$(3). We will show that (3)$\Rightarrow$(1) and (1)$\Rightarrow$(2).

(3)$\Rightarrow$(1): Lemma \ref{ringslice} states that for each maximal ideal $\mfp$ of $A^G$ and for any choice of $\mfP,\mfQ$ as in (3), 
\[
C_\mfp^{hs}\cong (A_\mfp^{hs})_\mfQ^{I_G(\mfP)}.
\]
Thus (3) implies that for each $\mfp$, $C_\mfp^{hs}$ is Cohen-Macaulay. The homomorphism of local noetherian rings
\[
(A^G)_\mfp \rightarrow C_\mfp^{hs}
\]
is flat, so by the result \cite[Theorem 2.1.7]{brunsherzog} quoted in \S \ref{sec:CMness}, Cohen-Macaulayness of $C_\mfp^{hs}$ is equivalent to that of $(A^G)_\mfp$ plus that of $C_\mfp^{hs}/\mfp C_\mfp^{hs}$. In particular, since $C_\mfp^{hs}$ is Cohen-Macaulay, so is $(A^G)_\mfp$. Since this holds for all maximal ideals $\mfp$ of $A^G$, $A^G$ is Cohen-Macaulay.

(1)$\Rightarrow$(2) Suppose $A^G$ is Cohen-Macaulay. Let $\mfp$ be any prime ideal of $A^G$. It suffices to prove that $C_\mfp^{hs}$ is Cohen-Macaulay, since by lemma \ref{ringslice}, for any $\mfP,\mfQ$ as in (2), we have
\[
C_\mfp^{hs}\cong (A_\mfp^{hs})_\mfQ^{I_G(\mfP)}.
\]
Since $(A^G)_\mfp \rightarrow C_\mfp^{hs}$ is flat, we again have by \cite[Theorem 2.1.7]{brunsherzog} that the Cohen-Macaulayness of $C_\mfp^{hs}$ is equivalent to that of $(A^G)_\mfp$ plus that of $C_\mfp^{hs}/\mfp C_\mfp^{hs}$. The former ring is Cohen-Macaulay since $A^G$ is, by the hypothesis (1), and the latter is Cohen-Macaulay since it is a field (cf. \S \ref{sec:CMness}), namely, the residue field of the local ring $C_\mfp^{hs}$.
\end{proof}

Theorem \ref{m} allows us to test Cohen-Macaulayness of an invariant ring $A^G$ locally, prime-by-prime, in terms of  the local ring $(A_\mfp^{hs})_\mfQ$ and the local group action $I_G(\mfP)$. For the application we have in mind in \S \ref{perminvars}, we will need to carry information about $A$ and $G$ to $(A_\mfp^{hs})_\mfQ$ and $I_G(\mfP)$, so we enunciate a few more lemmas to accomplish this:

\begin{lemma}\label{lem:AshisCM}
If $A$ is Cohen-Macaulay, then for any prime ideal $\mfp$ of $A^G$, $A_\mfp^{hs}$ is Cohen-Macaulay.
\end{lemma}

\begin{proof}
Suppose $A$ is Cohen-Macaulay, thus noetherian, and $\mfp$ is a prime of $A^G$. By lemma \ref{lem:noetherian}, $A_\mfp^{hs}$ is noetherian.

Let $\mfQ$ be any maximal ideal of $A_\mfp^{hs}$ and let $\mfP$ be its contraction in $A$. (Note that $\mfQ$ lies over $\mfp C_\mfp^{hs}$, per section \ref{sec:groupactions}, and therefore $\mfP$ lies over $\mfp$.) Now
\[
A^G\rightarrow (A^G)_\mfp \rightarrow C_\mfp^{hs}
\]
is a flat map. Therefore, base changing by $A^G\rightarrow A_\mfP$,
\[
A_\mfP\rightarrow A_\mfP\otimes_{A^G} C_\mfp^{hs} = A_\mfP \otimes_A A_\mfp^{hs}
\]
is also a flat map. Since $\mfQ\triangleleft A_\mfp^{hs}$ pulls back to $\mfP$ in $A$, $(A_\mfp^{hs})_\mfQ$ is a localization of $A_\mfP \otimes_A A_\mfp^{hs}$; thus
\[
A_\mfP\rightarrow (A_\mfp^{hs})_\mfQ
\]
is also flat. Therefore, again by \cite[Theorem 2.1.7]{brunsherzog} discussed in \S \ref{sec:CMness}, Cohen-Macaulayness of $(A_\mfp^{hs})_\mfQ$ is equivalent to that of $A_\mfP$ plus that of $(A^{hs}_\mfp)_\mfQ / \mfP (A^{hs}_\mfp)_\mfQ$. The former is Cohen-Macaulay since $A$ is, while the latter is Cohen-Macaulay since it is an artinian local ring (cf. \S \ref{sec:CMness}), which in turn is because $A_\mfP\rightarrow (A_\mfp^{hs})_\mfQ$ is of relative dimension zero. This itself is because this map is a localization of the base change $A_\mfP\otimes_{(A^G)_\mfp}-$ of the map $(A^G)_\mfp\rightarrow C_\mfp^{hs}$, which is flat of relative dimension zero because it is a strict henselization (\cite[Proposition 18.8.8(iii)]{grothendieck}).
\end{proof}



For a natural number $t$, an element $g\in G$ is called a \textbf{$t$-reflection} if the ideal generated by $(g-1)A$ in $A$ is contained in a prime of height $\leq t$.  A prime $\mfP$ contains $(g-1)A$ if and only if $g\in I_\mfP(A)$, so another way to say this is that $g$ is a $t$-reflection if it is in the inertia group of some prime of height $\leq t$.

In the geometric situation (where $A$ is a finitely generated algebra over a field), the ideal generated by $(g-1)A$ corresponds to the fixed point locus of $g$, so this definition makes a group element a $t$-reflection if this fixed point locus has codimension at most $t$. Thus if $G$ is a linear group acting on the coordinate ring of affine space, a $1$-reflection is either the identity or a reflection in the classical sense. A $2$-reflection has a fixed point locus of codimension 0, 1, or 2. In particular, if $G$ acts by permutations of a basis, then the $2$-reflections are exactly the identity, the transpositions, the double transpositions, and the 3-cycles.

\begin{lemma}\label{lem:treflectionlocal}
If an element $g\in I_G(\mfP)$ acts as a $t$-reflection on $A_\mfP$, then it acts as a $t$-reflection on $A$.
\end{lemma}

\begin{proof}
Since $g\in I_G(\mfP)$, we have $(g-1)A\subset \mfP$. The primes of $A$ contained in $\mfP$ are in containment-preserving bijection with the primes of $A_\mfP$, with the bijection given by extension along the canonical localization map, and $(g-1)A_\mfP$ is the extension of $(g-1)A$ along this map. Thus if a prime of height $t$ in $A_\mfP$ contains $(g-1)A_\mfP$, then its pullback in $A$ is also of height $t$ and contains $(g-1)A$.
\end{proof}

\begin{lemma}\label{lem:treflectionsh}
If $A$ is noetherian, and $g\in I_G(\mfP)=I_G(\mfQ)$ acts as a $t$-reflection on $A_\mfp^{hs}$, then it acts as a $t$-reflection on $A$.
\end{lemma}

\begin{proof}
If $g$ is a $t$-reflection on $A_\mfp^{hs}$, then there is a prime ideal $\mathfrak{S}$ of $A_\mfp^{hs}$ of height $\leq t$ and containing $(g-1)A_\mfp^{hs}$. Let $\mathfrak{R}$ be $\mathfrak{S}$'s pullback in $A$. Then $\mathfrak{R}$ contains $(g-1)A$. Since by section \ref{sec:groupactions} and lemma \ref{lem:noetherian},
\[
A\rightarrow A_\mfp^{hs}
\]
is a flat extension of noetherian rings, going-down applies (\cite[Lemma 10.11]{eisenbud}), so that the height of $\mathfrak{S}$ is at least that of $\mathfrak{R}$. In particular, the height of $\mathfrak{R}$ is $\leq t$, so that $g$ is a $t$-reflection on $A$.
\end{proof}

We will also need to take an element of $G$ acting on $A$ but not as a $t$-reflection, and conclude that it does not act on a certain subring as a $t$-reflection either:

\begin{lemma}\label{lem:treflectionN}
If $N$ is the normal subgroup of $G$ generated by the $t$-reflections, then no element of $G\setminus N$ acts on $A^N$ as a $t$-reflection.
\end{lemma}

\begin{remark}
This lemma does not require a noetherian hypothesis on $A$.
\end{remark}

\begin{proof}
Let $g\in G$. We will show that if its image $\overline{g}\in G/N$ acts on $A^N$ as a $t$-reflection, then actually $g\in N$.

If $\overline{g}$ acts on $A^N$ as a $t$-reflection, then there is a prime $\mfp$ of $A^N$ of height $\leq t$ with $\overline{g}\in I_{G/N}(\mfp)$. Let $\mfP$ be any prime of $A$ lying over $\mfp$. The height of $\mfP$ is equal to that of $\mfp$ (e.g. by \cite[Lemma 5.3]{gordeevkemper}, which is stated for noetherian $A$ but the argument holds in general); in particular it is $\leq t$. By lemma \ref{lem:GNinertia}, we have
\[
I_{G/N}(\mfp) = I_G(\mfP)/I_N(\mfP).
\]
In particular, $I_G(\mfP)$ surjects onto $I_{G/N}(\mfp)$, so there is an element $g'\in I_G(\mfP)$ whose image in $G/N$ is $\overline{g}$. Since $\mfP$ has height $\leq t$, $g'$ is a $t$-reflection, so it is contained in $N$ by construction. Then its image $\overline{g}$ must actually be the identity. So $g$ (with the same image) lies in the kernel of $G\rightarrow G/N$, i.e. $g\in N$.
\end{proof}

The following lemma allows us to detect a failure of Cohen-Macaulayness locally.


\begin{lemma}\label{lem:GKLP}
Let $A$ be a ring containing the prime field $\FF_p$, and let $G$ be a $p$-group. Suppose that $A$ is Cohen-Macaulay, $A^G$ is noetherian, and $A$ is finite over $A^G$. Further, suppose there is a prime ideal $\mfP$ of $A$ such that $G=I_G(\mfP)$. Then $A^G$ is not Cohen-Macaulay unless $G$ is generated by its $2$-reflections.
\end{lemma}

\begin{remark}
This statement is closely related to \cite[Theorem 5.5]{gordeevkemper}, which also applies to non-$p$-groups and gives some control over how far $A^G$ can be from Cohen-Macaulay. However, a key step in the proof of that result requires the rings to be normal rings that are localizations of algebras finitely generated over fields. As our application will be to rings that do not fulfill this hypothesis, we give an independent proof.
\end{remark}

\begin{proof}[Proof of lemma \ref{lem:GKLP}]
Let $N$ be the normal subgroup of $G$ generated by the $2$-reflections. 

Since $A$ is finite over the noetherian ring $A^G$, it is noetherian as an $A^G$-module. Since it also contains $\FF_p$, \cite[Corollary 4.3]{lorenzpathak} applies, which, when specialized to the situation that $G$ is a $p$-group, states that if both $A$ and $A^G$ are Cohen-Macaculay, then the map
\[
\Tr_{G/N}: A^N \rightarrow A^G
\]
given by
\[
x\mapsto \sum_{g\in G/N} gx
\]
is surjective onto $A^G$, where we think of each $g$ as an element of $G$ and the sum is taken over coset representatives of $N$.

We will show that this map cannot be surjective unless $N=G$. Since $A$ is Cohen-Macaulay by assumption, this will show $A^G$ is not Cohen-Macaulay if $N\neq G$.

If $\Tr_{G/N}$ is surjective, then we have
\[
1 = \sum_{g\in G/N} gx
\]
for some $x\in A^N$. Since $G=I_G(\mfP)$, all $g\in G$ satisfy $gx = x \mod \mfP$ in $A$, thus
\[
1 = \sum_{g\in G/N} x = [G:N] x\mod \mfP
\]
in $A$. Since $G$ is a $p$-group and $A$ contains $\FF_p$, $[G:N] x = 0$ in $A$ unless $N=G$. In particular, $[G:N]x$ cannot be $1\mod\mfP$ unless $N=G$.
\end{proof}

\begin{remark}
The map $\Tr_{G/N}$ is called the {\em relative trace} or {\em relative transfer}; see remark \ref{rmk:transfer}.
\end{remark}

\begin{remark}
The proof uses the result \cite[Lemma 4.3]{lorenzpathak} of Lorenz and Pathak, which has as a hypothesis that $A$ is noetherian as an $A^G$-module; call this ($\star$). Above, we deduced ($\star$) from the assumptions that (1) $A^G$ is noetherian and (2) $A$ is finite over it. Actually, ($\star$) also implies (1) and (2), hence is equivalent to them. Since any ideal of $A^G$ is also an $A^G$-submodule of $A$ (since $A^G$ embeds in $A$), ($\star$) implies that all these ideals are finitely generated, thus (1). Meanwhile, $A$ itself is an $A^G$-submodule of $A$, so ($\star$) implies it is finitely generated as an $A^G$-module, thus (2). More generally, if a module $M$ over a ring $R$ has an injective $R$-module map from $R$, then noetherianity of $M$ as $R$-module is equivalent to noetherianity of $R$ as a ring plus finite generation of $M$ over $R$, by the same arguments.
\end{remark}

Combining all of these results, we get an obstruction to Cohen-Macaulayness for a characteristic $p$ ring expressed entirely in terms of the presence of a certain inertia group. The proof of the ``only-if" direction of theorem \ref{mainresult} will be an application of this proposition.

\begin{proposition}\label{prop:pobstruction}
Let $A$ be a ring containing $\FF_p$ and let $G$ be a finite group of automorphisms of $A$. Let $N$ be the normal subgroup of $G$ generated by the 2-reflections. Suppose that $A^N$ is Cohen-Macaulay, $A^G$ is noetherian, and $A^N$ is finite over $A^G$. If there is an inertia group for the action of $G/N$ on $A^N$ that is a nontrivial $p$-group, then $A^G$ is not Cohen-Macaulay.
\end{proposition}

\begin{proof}
Note that 
\[
A^G = (A^N)^{G/N}.
\]

Since $A^G$ is noetherian, theorem \ref{m} applies.

Suppose $\mfP$ is a prime of $A^N$ whose inertia group $I_{G/N}(\mfP)$ is a $p$-group, per the hypothesis. Let 
\[
\mfp = \mfP \cap (A^N)^{G/N},
\]
let $C_\mfp^{hs}$ be the strict henselization of $(A^G)_\mfp = ((A^N)^{G/N})_\mfp$, and let
\[
(A^N)_\mfp^{hs} = A^N \otimes_{A^G} C_\mfp^{hs},
\]
as in section \ref{sec:groupactions}.

By assumption, $A^N$ is Cohen-Macaulay. Thus $(A^N)_\mfp^{hs}$ is Cohen-Macaulay, by lemma \ref{lem:AshisCM}, and thus so is 
\[
((A^N)_\mfp^{hs})_\mfQ
\]
for any $\mfQ\triangleleft (A^N)_\mfp^{hs}$, and in particular any $\mfQ$ as described in theorem \ref{m}.

As $A^N$ is finite over the noetherian ring $A^G$ by assumption, its base change $(A^N)_\mfp^{hs}$ is finite over $C_\mfp^{hs}$, which is noetherian by \cite[Proposition 18.8.8(iv)]{grothendieck}, as discussed in section \ref{sec:groupactions}. The localization $((A^N)_\mfp^{hs})_\mfQ$ is a homomorphic image of $(A^N)_\mfp^{hs}$ by the isomorphism \eqref{iso}, so it too is finite over $C_\mfp^{hs}$.

By lemma \ref{ringslice}, $C_\mfp^{hs}$ is the invariant ring for the action of $I_{G/N}(\mfP)$ on $((A^N)_\mfp^{hs})_\mfQ$. Since $A$ contains $\FF_p$ and therefore so do $A^N$ and $((A^N)_\mfp^{hs})_\mfQ$, and since $I_{G/N}(\mfP)$ is a $p$-group that is equal to $I_{G/N}(\mfQ)$ which is an inertia group of $((A^N)_\mfp^{hs})_\mfQ$, we have now verified all the hypotheses of lemma \ref{lem:GKLP} for the action of $I_{G/N}(\mfP)$ on $((A^N)_\mfp^{hs})_\mfQ$. We can conclude from that lemma that the invariant ring cannot be Cohen-Macaulay unless $I_{G/N}(\mfP)$ is generated by 2-reflections. 

However $I_{G/N}(\mfP)$ is not so generated. By lemma \ref{lem:treflectionN}, no nontrivial element of $G/N$ acts on $A^N$ as a 2-reflection. In particular, no nontrivial element of $I_{G/N}(\mfP)$ acts on $A^N$ as a 2-reflection. Since $A^N$ is Cohen-Macaulay, it is noetherian, so lemma \ref{lem:treflectionsh} applies, and no nontrivial element of $I_{G/N}(\mfP)$ acts on $(A^N)_\mfp^{hs}$ as a 2-reflection either. By lemma \ref{lem:treflectionlocal}, the same is true for the action of $I_{G/N}(\mfP) = I_{G/N}(\mfQ)$ on
\[
((A^N)_\mfp^{hs})_\mfQ.
\]
In particular, the $p$-group $I_{G/N}(\mfP)$ is not generated by $2$-reflections on this ring, since it is nontrivial. Then lemma \ref{lem:GKLP} implies that
\[
((A^N)_\mfp^{hs})_\mfQ^{I_{G/N}(\mfP)}
\]
is not Cohen-Macaulay. Therefore, by theorem \ref{m}, neither is $(A^N)^{G/N} = A^G$.
\end{proof}

\section{Permutation invariants}\label{perminvars}

In this section we prove the two directions of theorem \ref{mainresult}. A schematic diagram of the proof is found in figure \ref{fig:schematic}.

\begin{figure}
\begin{center}
\begin{tikzpicture}

\node (Ggenby) at (0,3) {$G$ gen by 2-reflections};
\node (DeltaQuotient) at (0,6) {$\Delta / G$ satisfies \eqref{eq:homvanishing}};
\node (lange) at (0,4.5) {\scriptsize Lange's theorem (\S \ref{sec:if})};
\draw (Ggenby) -- (lange);
\draw [->] (lange) -- (DeltaQuotient);

\node (SRCM) at (5,6) {$k[\Delta / G]$ is CM};
\draw [<->] (DeltaQuotient) -- (SRCM);
\node at (2.6,6.2) {\scriptsize Reisner etc. (\S \ref{sec:CCAandPL})};

\node (k[delta]^G) at (10,6) {$k[\Delta]^G$ is CM};
\draw [<->] (SRCM) -- (k[delta]^G);
\node at (7.5, 6.2) {\scriptsize Reiner (\S \ref{sec:CCAandPL})};

\node (RGCM) at (10,3) {$k[\mathbf{x}]^G$ is CM};
\node (GSR) at (10,4.5) {\scriptsize Garsia-Stanton/Reiner (\S \ref{sec:CCAandPL})};
\draw (k[delta]^G) -- (GSR);
\draw [->] (GSR) -- (RGCM);

\node (Gnotgenby) at (0,0) {$G$ not gen by 2-reflections};
\node (pexists) at (0,-3) {$\exists p$ with $G_\pi^BN/N \cong \ZZ/p$};
\node (primestab) at (0,-1.5) {\scriptsize Lemma \ref{lem:primestabilizer} (\S \ref{sec:onlyif})};
\draw (Gnotgenby) -- (primestab);
\draw [->] (primestab) -- (pexists);

\node (Iexists) at (5,-6) {$\exists \mfP \triangleleft k[\mathbf{x}]^N$ s.t. $I_{G/N}(\mfP) \cong \ZZ/p$};
\node (inertiais) at (2.5, -4.5) {\scriptsize Lemma \ref{lem:inertiais} (\S \ref{sec:onlyif})};
\draw (pexists) -- (inertiais);
\draw [->] (inertiais) -- (Iexists);

\node (prop38) at (7.5,-3) {\scriptsize Proposition \ref{prop:pobstruction} (\S \ref{sec:inertiaCM})};
\draw (Iexists) -- (prop38);

\node (NisCM) at (5,-2) {$k[\mathbf{x}]^N$ is CM};
\draw (-1.7,1.5) -- (11,1.5);
\draw (-1.7,2) -- (-1.7,1.5);
\draw (11,2) -- (11,1.5);
\draw [->] (5,1.5) -- (NisCM);
\draw (NisCM) -- (prop38);

\node (RGnotCM) at (10,0) {If $\Char k = p$, $k[\mathbf{x}]^G$ is not CM};
\draw [->] (prop38) -- (RGnotCM);

\end{tikzpicture}
\end{center}
\caption{Schematic diagram of the proof of theorem \ref{mainresult}. Arrows are implications, and small print above or interrupting an arrow names a result needed for the implication to go through. The \S-references indicate where to look for statements and notation definitions. The top half is the ``if" direction (proposition \ref{prop:ifdirection}). The bottom half is the ``only-if" direction (proposition \ref{prop:onlyifdirection}). The group $N$ is the subgroup of $G$ generated by the 2-reflections, so the ``if" direction is required to conclude that $k[\mathbf{x}]^N$ is Cohen-Macaulay in the bottom half.}\label{fig:schematic}
\end{figure}
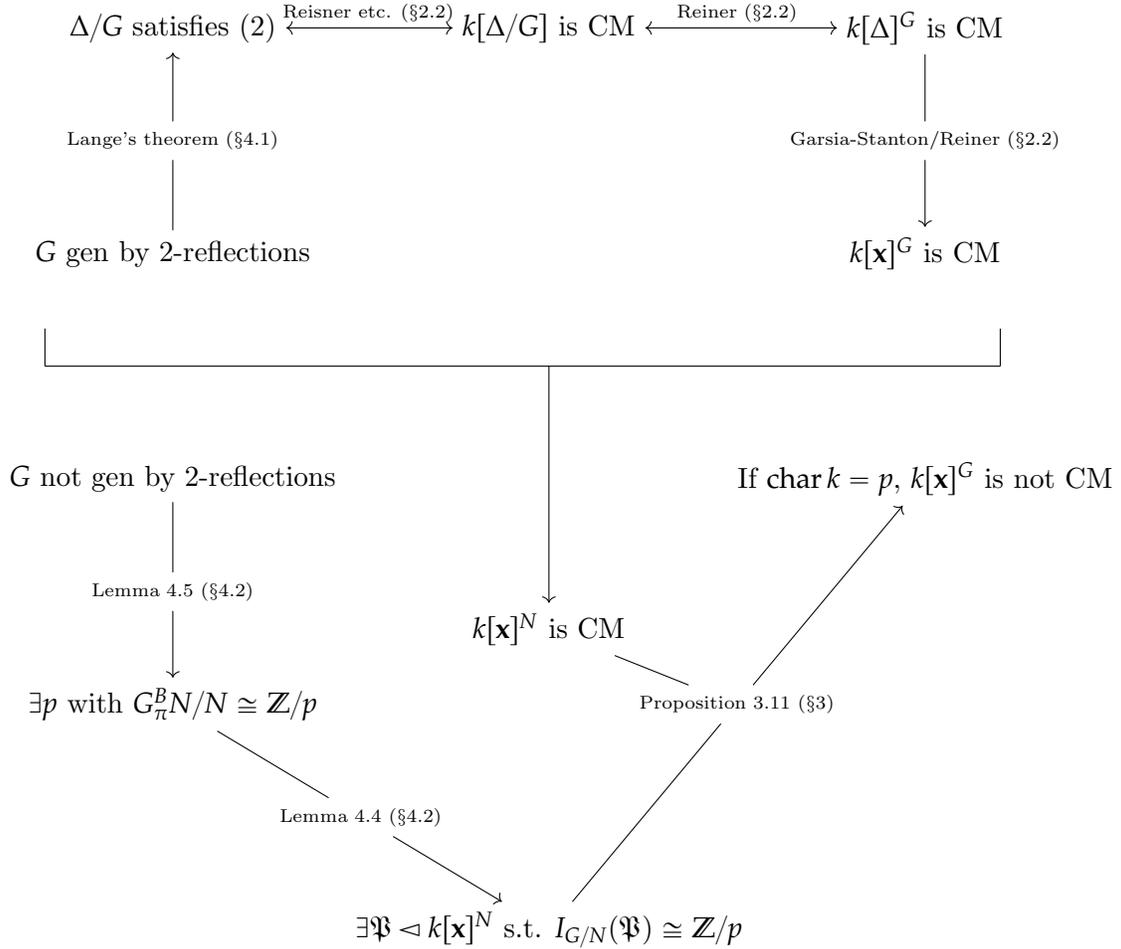

\subsection{The {\em if} direction}\label{sec:if}

In this section we prove:

\begin{proposition}\label{prop:ifdirection}
If $G$ is generated by transpositions, double transpositions, and 3-cycles, then $k[\mathbf{x}]^G$ is Cohen-Macaulay regardless of the field $k$.
\end{proposition}

The groundwork has been laid in \S\ref{sec:CCAandPL}. The remaining piece of the proof is supplied by a recent, beautiful result of Christian Lange, building on earlier work of Marina Mikha\^{i}lova. Let $H$ be a finite subgroup of the orthogonal group $O_d(\RR)$, acting on $\RR^d$. Endow $\RR^d$ with its standard piecewise-linear (PL) structure. The topological quotient $\RR^d/H$ carries a PL structure such that the quotient map $\RR^d\rightarrow \RR^d/H$ is a PL map, and the main result of \cite{lange2} is that it is a PL manifold (possibly with boundary) if and only if $H$ is generated by $2$-reflections. (Lange calls elements of $O_d(\RR)$ fixing a codimension-2 subspace {\em rotations} since they rotate a plane and fix its orthogonal complement, so he calls groups generated this way {\em rotation-reflection groups}.) The bulk of the work in this result lies in the ``if" direction. The proof is a delicate induction on the group order, based on a complete classification of rotation-reflection groups. This classification was proven in joint work with Marina Mikha\^{i}lova (\cite{langemikhailova}). 


\begin{proof}[Proof of proposition \ref{prop:ifdirection}] 
Let $G$ act on $\RR^n$ by permutations of the axes. Let $x_1,\dots,x_n$ be the coordinates on $\RR^n$. The subspace
\[
T = \left\{\sum_{i=1}^n x_i = 0 \right\}
\]
is $G$-invariant. Transpositions in $G$ act as reflections on $T$, while double transpositions and 3-cycles act as rotations. Thus under the hypothesis of the proposition, $G$ acts on $T$ as a rotation-reflection group. By Lange's work (\cite{lange2}), $T/G$ is a PL manifold.

Recall the $\Delta$ of section \ref{sec:CCAandPL}: it is the order complex of $B_n\setminus\{\emptyset\}$, which is the first barycentric subdivision of an $(n-1)$-simplex. Embed the underlying topological space $|\Delta|$ of $\Delta$ in $T$ as follows. First, map the vertices of $\Delta$ to the barycenters of the standard simplex
\[
\left\{x_i \geq 0,\; \sum x_i = 1\right\}
\]
in $\RR^n$ by mapping each vertex, which by definition is an element $\alpha \in B_n\setminus\{\emptyset\}$, which is itself a nonempty subset of $[n]$, to the barycenter 
\[
\frac{1}{|\alpha|}\sum_{i\in \alpha} e_i
\]
of the set of standard basis vectors $\{e_i\}_{i\in \alpha}$ corresponding to that subset. Then, extend this map to all of $\Delta$ by extending linearly from the vertices to each simplex in $\Delta$. Finally, project the affine hyperplane plane $\left\{\sum x_i = 1\right\}$ containing the image orthogonally onto $T$ via $(x_1,\dots,x_n)\mapsto (x_1-1/n,\dots,x_n-1/n)$. This embedding is $G$-equivariant for the action of $G$ on $|\Delta|$ induced from its action on $[n]$, and the present action of $G$ on $T$.

The embedded complex $|\Delta|\subset T$ is evidently a polyhedron, and it is a star of the origin in $T$ since it is the union of closed line segments from the origin to its compact boundary, these segments are disjoint except for the origin itself, and it is a neighborhood of the origin in $T$ (see the definition of a star in \S \ref{sec:CCAandPL}). Since the action of $G$ is linear, it permutes these segments. Thus $|\Delta| / G = |\Delta / G|$ is also a union of line segments from the (image of the) origin to its compact boundary, and these segments are disjoint except for the origin itself. Also, $|\Delta / G|$ is a neighborhood of the (image of the) origin since $T \rightarrow T / G$ is the quotient map by a group of homeomorphisms and is therefore an open map. It is additionally a polyedron since the quotient map $T\rightarrow T/G$ is PL, and the image of a compact polyhedron under a PL map is a compact polyhedron (\cite[Corollary 2.5]{rourkesanderson}). In other words, $|\Delta / G|$ is a polyhedral star of the image of the origin in the PL $(n-1)$-manifold $T/G$. It is therefore (per \cite[pp.~20-21]{rourkesanderson}, see the discussion at the end of \S \ref{sec:CCAandPL}) homeomorphic to a ball. In particular, it is contractible, thus
\[
\tilde H_i(|\Delta / G|;k) = 0
\]
for all $i$, regardless of the field $k$; and it is a manifold (with boundary), thus
\[
H_i(|\Delta / G|, |\Delta /G| - q; k) = 0
\]
for all $i < n-1$ and all $q\in |\Delta/G|$, regardless of $k$. Thus it satisfies \eqref{eq:homvanishing} for all $i<\dim \Delta / G$ and all $q\in |\Delta/G|$, so by the discussion in \S \ref{sec:CCAandPL}, $k[\mathbf{x}]^G$ is Cohen-Macaulay.
\end{proof}

\subsection{The {\em only-if} direction}\label{sec:onlyif}

In this section we complete the proof of theorem \ref{mainresult} by proving the converse of \ref{prop:ifdirection}.

\begin{proposition}\label{prop:onlyifdirection}
If $G$ is not generated by transpositions, double transpositions, and 3-cycles, then there exists a prime $p$ such that for any $k$ of characteristic $p$, $k[\mathbf{x}]^G$ is not Cohen-Macaulay.
\end{proposition}

The proof is at the end of the section. Actually we prove somewhat more: for a group $G$ not generated by transpositions, double transpositions, and 3-cycles, we give an explicit construction yielding the prime $p$. The precise statement is given below as proposition \ref{prop:onlyifconstructive}.

In this section, $p$ is conceptually prior to the field $k$. Our proof will first construct $p$ and then prove that when $\Char k = p$, $k[\mathbf{x}]^G$ is not Cohen-Macaulay.

We develop the needed machinery for the proof. Let $\Pi_n$ be the poset of partitions of the set $[n]$, with the order relation given, for any $\pi , \tau \in \Pi_n$, by 
\[
\pi\leq \tau \Leftrightarrow \pi \text{ refines } \tau.
\]
An element $g \in G\subset S_n$ partitions $[n]$ into orbits, and thus determines an element $\pi \in \Pi_n$. This gives a map
\begin{align*}
\varphi:G &\rightarrow \Pi_n\\
g &\mapsto \pi.
\end{align*}

If $\pi \in \Pi_n$, we write $G_\pi^B$ for the blockwise stabilizer of $\pi$ in $G$, i.e. the set of elements of $G$ that act separately on each block of $\pi$.

For a given $\pi\in\Pi_n$, let $\mfP_\pi^\star$ be the prime ideal of $k[\mathbf{x}]$ generated by the binomials $x_i-x_j$ for every pair $i,j\in [n]$ lying in the same block of $\pi$. The dimension of $\mfP_\pi^\star$ (i.e. the dimension of $k[\mathbf{x}]/\mfP_\pi^\star$) is the number of blocks of $\pi$.

\begin{lemma}\label{lem:inertiaupstairsis}
With this notation, we have
\[
I_G(\mfP_\pi^\star) = G_\pi^B.
\]
\end{lemma}

\begin{proof}The ring $k[\mathbf{x}]/\mfP_\pi^\star$ is the polynomial ring obtained by identifying $x_i$ with $x_j$ for each $i,j$ in the same block of $\pi$, so its indeterminates are in bijection with the blocks of $\pi$. If $h\in G_\pi^B$, then $h$ acts separately on the $x_i$'s in each block, and therefore $h$ fixes $\mfP_\pi^\star$ setwise and the induced action on $k[\mathbf{x}]/\mfP_\pi^\star$ is trivial. Thus $h\in I_G(\mfP_\pi^\star)$. Conversely, if $h\notin G_\pi^B$, then either $h$ fixes $\pi$ but not blockwise, in which case $h$ fixes $\mfP_\pi^\star$ setwise but the action of $h$ on $k[\mathbf{x}]/\mfP_\pi^\star$ is not trivial, so that $h\in D_G(\mfP_\pi^\star)$ but not $I_G(\mfP_\pi^\star)$; or else $h$ does not fix $\pi$ at all, in which case it does not act on $\mfP_\pi^\star$, and is not contained in $D_G(\mfP_\pi^\star)$, let alone $I_G(\mfP_\pi^\star)$.
\end{proof}

If $N$ is a normal subgroup of $G$, denote by $G_\pi^BN/N$ the image of $G_\pi^B$ in the quotient $G/N$, and let
\[
\mfP_\pi = \mfP_\pi^\star \cap k[\mathbf{x}]^N.
\]

\begin{lemma}\label{lem:inertiais}
With this notation, we have
\[
I_{G/N}(\mfP_\pi) = G_\pi^BN/N.\qedhere
\]
\end{lemma}

\begin{proof}
We have from lemma \ref{lem:GNinertia} that
\[
I_{G/N}(\mfP_\pi) = I_G(\mfP_\pi^\star)/I_N(\mfP_\pi^\star) = I_G(\mfP_\pi^\star) / \left(N\cap I_G(\mfP_\pi^\star)\right),
\]
and from lemma \ref{lem:inertiaupstairsis} that
\[
I_G(\mfP_\pi^\star) / (N\cap I_G(\mfP_\pi^\star))=G_\pi^B/ (N\cap G_\pi^B) = G_\pi^BN/N.\qedhere
\]
\end{proof}

The following lemma is the device we use to find the characteristic $p$ in which we can prove that $k[\mathbf{x}]^G$ fails to be Cohen-Macaulay.

\begin{lemma}\label{lem:primestabilizer}
Let $N\triangleleft G$ be a proper normal subgroup. Let $\pi$ be minimal in $\Pi_n$ among partitions associated (via $\varphi$) with elements of $G$ that are not in $N$. Then:
\begin{enumerate}
\item The group $G_\pi^BN/N$ is cyclic of prime order, say $p$; \label{concl:cyclic}
\item any element $g$ of $G\setminus N$ whose orbits are given by $\pi$ has order a power of $p$, and \label{concl:ppowerorder}
\item the image of $g$ in $G/N$ generates $G_\pi^BN/N$. \label{concl:genbyg}
\end{enumerate}
\end{lemma}
\begin{proof}
Let $g$ be an element of $G\setminus N$ whose orbits are given by $\pi$, and let $h$ be any other nontrivial element of $G_\pi^B$, in other words a nontrivial element of $G$ whose orbits refine $\pi$. (Note that, by minimality of $\pi$, either $\varphi(h) = \pi$ or else $h\in N$.) Pick any element $a\in [n]$ acted on nontrivially by $h$. Then $g$ acts nontrivially on $a$ as well since $h$'s orbits refine $g$'s.

Since $h$ preserves $\pi$ and $g$ acts transitively on each block of $\pi$, there is an $m\in\ZZ$ such that $g^m(a) = h(a)$. Then $h^{-1}g^m(a) = a$, so that $h^{-1}g^m$ both preserves $\pi$ and has a fixed point $a$ that $g$ does not have. Thus its orbits properly refine $\pi$, and minimality of $\pi$ among partitions associated to elements of $G\setminus N$ implies that $h^{-1}g^m\in N$. Thus $hN = g^mN$. This shows that $g$ generates the image of $G_\pi^B$ in $G/N$, proving \ref{concl:genbyg}; thus $G_\pi^BN/N$ is cyclic. Meanwhile, for any prime $p$ dividing the order of $g$, $g^p$'s orbits also properly refine $g$'s, so $g^p$ is in $N$ too; thus $g$'s image in $G/N$ has order dividing $p$. Since $g\notin N$ by construction, the order of $g$'s image in $G/N$ is exactly $p$. This completes the proof of \ref{concl:cyclic}. If $q$ is a hypothetical second prime dividing the order of $g$ in $G$, then the order of $g$'s image in $G/N$ is $q$, for the same reason it is $p$, and it follows that $q=p$ after all, so there is no such second prime. Therefore $g$ has $p$-power order in $G$. This proves \ref{concl:ppowerorder}.
\end{proof}

\begin{proof}[Proof of proposition \ref{prop:onlyifdirection}]
Let $N$ be the subgroup of $G$ generated by the transpositions, double transpositions, and 3-cycles (i.e. $2$-reflections). By proposition \ref{prop:ifdirection}, $k[\mathbf{x}]^N$ is a Cohen-Macaulay ring. Since $k[\mathbf{x}]$ is a finitely generated algebra over $k$, $k[\mathbf{x}]^G$ is also finitely generated as an algebra over $k$ (\cite[Chapitre V \S 1.9, Th\'{e}or\`{e}me 2]{bourbaki}), so in particular it is noetherian. By the same logic, $k[\mathbf{x}]^N$ is finitely generated as an algebra over $k$, and therefore over $k[\mathbf{x}]^G$. Since it is a subring of $k[\mathbf{x}]$, which is integral over $k[\mathbf{x}]^G$ by \cite[Chapitre V \S 1.9, Proposition 22]{bourbaki}, it is integral over $k[\mathbf{x}]^G$ as well, which, together with finite generation as an algebra, implies it is actually finite over the noetherian ring $k[\mathbf{x}]^G$. Thus if $k$ is a field of positive characteristic $p$, then proposition \ref{prop:pobstruction} applies, and we can show $k[\mathbf{x}]^G$ is not Cohen-Macaulay by exhibiting an inertia group for the action of $G/N$ on $k[\mathbf{x}]^N$ that is a nontrivial $p$-group.

Now if $N$ is a proper subgroup of $G$ per the hypothesis, then we can find a $\pi \in \Pi_n$ that is minimal among all partitions associated (via $\varphi$) with elements of $G\setminus N$. Then lemma \ref{lem:primestabilizer} gives us a prime number $p$ such that $G_\pi^BN/N$ is cyclic of order $p$, and then lemma \ref{lem:inertiais} gives us a prime ideal $\mfP_\pi$ of $k[\mathbf{x}]^N$ such that 
\[
I_{G/N}(\mfP_\pi) = G_\pi^BN/N.
\]
Thus, for any $k$ of this specific characteristic, we can conclude by proposition \ref{prop:pobstruction} that $k[\mathbf{x}]^G$ fails to be Cohen-Macaulay.
\end{proof}

An examination of the proof in view of conclusion \ref{concl:ppowerorder} of lemma \ref{lem:primestabilizer} shows that we have actually proven the following constructive version of proposition \ref{prop:onlyifdirection} with no additional work:

\begin{customthm}{\ref*{prop:onlyifdirection}b}\label{prop:onlyifconstructive}
Let $N$ be the subgroup of $G$ generated by the transpositions, double transpositions, and 3-cycles. If $N\subsetneq G$, then for any $g \in G\setminus N$ whose orbits are not refined by the orbits of any other $g\in G\setminus N$, the order of $g$ is a prime power $p^\ell$, where $p$ has the property that $k[\mathbf{x}]^G$ is not Cohen-Macaulay if $\Char k = p$.\qed
\end{customthm}

\section{Conclusion and further questions}\label{sec:concl}

In this section we note some implications of the results above, and pose questions for further exploration. Throughout, let $N$ be the subgroup of $G\subset S_n$ generated by the transpositions, double transpositions, and 3-cycles, as at the end of \S \ref{sec:onlyif}.

\subsection{Bad primes; relation to previous work}

Given a permutation group $G\subset S_n$, let us refer to the set of prime numbers $p$ for which, if $\Char k = p$, then $k[\mathbf{x}]^G$ fails to be Cohen-Macaulay, as $G$'s {\em bad primes}.

It was mentioned in the introduction that the ``if" direction of theorem \ref{mainresult} implies that $G$'s bad primes are a subset of the primes dividing $[G:N]$. We see this as follows: the ``if" direction implies that $k[\mathbf{x}]^N$ is Cohen-Macaulay. Then, since
\[
k[\mathbf{x}]^G = (k[\mathbf{x}]^N)^{G/N},
\]
it follows from the Hochster-Eagon theorem (\cite[Proposition 13]{hochstereagon}) that $k[\mathbf{x}]^G$ is Cohen-Macaulay in any characteristic not dividing the order of $G/N$. Meanwhile, the ``only-if" direction of theorem \ref{mainresult} implies that if the set of primes dividing $[G:N]$ is nonempty, then so is $G$'s set of bad primes.

It was also mentioned in the introduction that the present work unites and generalizes several previously known results: Reiner's (\cite{reiner92}) theorem that the invariant rings of $A_n$ and the diagonally embedded $S_n\hookrightarrow S_n\times S_n$ are Cohen-Macaulay over all fields; Hersh's (\cite{hersh}, \cite{hersh2}) similar theorem for the wreath product $S_2\wr S_n\subset S_{2n}$, and Kemper's (\cite{kemper99}) theorems that in the $p$-group case, the ``only-if" direction of theorem \ref{mainresult} holds, and that the invariant ring of a {\em regular} permutation group $G$ is Cohen-Macaulay over all fields if and only if $G=C_2$, $C_3$, or $C_2\times C_2$, and in all other cases, every prime dividing $|G|$ is a bad prime for $G$. Most of these results are immediate implications of the ``if" direction of theorem \ref{mainresult}:
\begin{itemize}
\item The group $A_n$ is generated by 3-cycles.
\item The diagonal $S_n\hookrightarrow S_n\times S_n$ is generated by the double transpositions $(i,i+1)(i+n,i+n+1)$ for $i=1,\dots,n-1$.
\item The wreath product $S_2\wr S_n$ is generated by the transpositions $(2i-1,2i)$ and the double transpositions $(2i-1,2i+1)(2i,2i+2)$ for $i=1,\dots,n-1$.
\item The regular representations of $C_2$, $C_3$, and $C_2\times C_2$ are generated by (in fact, their only nontrivial elements are) transpositions, 3-cycles, and double transpositions, respectively.
\end{itemize}
Recovering the other half of Kemper's result on regular permutation groups (that {\em every} prime dividing $|G|$ is bad for $G$) from the present work requires the constructive version of the ``only-if" direction given in proposition \ref{prop:onlyifconstructive}. Recall that if $G$ acts regularly, i.e. freely and transitively, on $[n]$, then this action is isomorphic to $G$'s left-translation action on its own elements. Then we have $|G| = n$, and every element $g$ of $G$ splits $[n]$ into orbits of equal length the order of $g$, because these orbits are in bijection with the right cosets $\langle g\rangle h$, $h\in G$.

If $G$ acts regularly and $|G| = n\geq 5$, then $G$ does not contain any transpositions, double transpositions, or 3-cycles, so $N$ is trivial. If $p$ is any prime dividing $|G|$, then $G$ has an element $g$ of order $p$, which, by the discussion in the last paragraph, partitions $[n]$ into orbits of equal length $p$. This partition cannot be refined by any nontrivial partition with parts of equal length since $p$ is prime; thus no element of $G\setminus N = G\setminus \{1\}$ can have orbits refining $g$'s. It follows from proposition \ref{prop:onlyifconstructive} that $p$ is a bad prime for $G$.

The remaining case is $n=4$ and $G=C_4$. In this case, $G$ is a $2$-group not generated by its lone double transposition, so it follows from theorem \ref{mainresult} that 2 is a bad prime for $G$.

\subsection{Groups generated by transpositions, double transpositions, and 3-cycles}

Theorem \ref{mainresult} calls attention to the family of permutation groups generated by transpositions, double transpositions, and 3-cycles. One may wonder how extensive is this family of groups. It turns out to be very limited. One can extract a classification from Lange and Mikha\^{i}lova's classification of all rotation-reflection groups (\cite{langemikhailova}), but this is more power than is needed. In the case that $G$ is transitive, such groups were already classified in 1979 by W. Cary Huffman (\cite[Theorem 2.1]{huffman}):
\begin{enumerate}
\item If $G$'s transpositions generate a transitive subgroup, then $G=S_n$.
\item If $G$ contains a transposition but the transpositions do not act transitively, then $n=2m$ is even and $G$ is isomorphic to the wreath product $S_2\wr S_m$.
\item If $G$ does not contain a transposition but does contain a three-cycle, then $G = A_n$.
\item Otherwise, $G$ contains no transpositions or 3-cycles and is generated by double transpositions. Then we have:
\begin{enumerate}
\item If $G$ contains a subgroup acting transitively on 5 points and fixing the rest, then either $n=5$ and $G\cong D_5$ in its usual action on the vertices of a regular pentagon, or else $n=6$ and $G\cong A_5\cong PSL(2,5)$ in its transitive action on $6$ points, e.g. the six points of the projective line over $\FF_5$.
\item If $G$ contains a subgroup acting transitively on 7 points and fixing the rest, then either $n=7$ and $G\cong GL(3,2)$ acting on the nonzero vectors of $\FF_2^3$, or else $n=8$ and $G\cong AGL(3,2) = \FF_2^3 \rtimes GL(3,2)$ acting on the points of $\AAA_{\FF_2}^3$.
\item If $G$ does not contain either of these kinds of subgroups, then $n=2m$ is even, and $G$ is isomorphic to the alternating subgroup of the wreath product $S_2\wr S_m$.
\end{enumerate}
\end{enumerate}
When one considers intransitive groups $G$, one does not end up too far beyond direct products of the above, since transpositions and 3-cycles can only act in a single orbit, while double transpositions can only act in two orbits, as a transposition in each. For example, if $G$ has two orbits, the classification begins as follows. If $G$ is not a direct product of the above, it contains a double transposition that acts as a transposition in each orbit. Then its image in each orbit contains a transposition, so is either $S_n$ or $S_2\wr S_m$ by the above. The possibilities are then highly constrained by Goursat's lemma.

Thus theorem \ref{mainresult} shows that most permutation groups $G$ have at least one bad prime.

\subsection{Further questions}

Since theorem \ref{mainresult} implies that the set of bad primes of $G$ is contained in the set of prime factors of $[G:N]$ and is nonempty exactly when the latter is nonempty, one might hope that these two sets are always equal. This is not the case. For example, let $G\subset S_7$ be the Frobenius group of order 21 generated by
\[
(1234567), (124)(365).
\]
All the nontrivial elements in this group are 7-cycles or double 3-cycles. Thus $N$ is trivial in this case, and the candidate bad primes are 3 and 7. 

Now $\pi = \{1,2,4\}\cup \{3,5,6\} \cup \{7\}$ is a minimal partition as in lemma \ref{lem:primestabilizer}, and thus the corresponding $g=(124)(365)$ generates an inertia group of order 3 for the action of $G/N = G$ on $k[\mathbf{x}]^N = k[\mathbf{x}]$. Then proposition \ref{prop:onlyifconstructive} shows that if $k$ has characteristic $3$, $k[\mathbf{x}]^G$ fails to be Cohen-Macaulay; i.e. 3 is a bad prime for this $G$.

On the other hand, 7 is not a bad prime for this $G$. This can be seen using the criterion given by Kemper in \cite[Theorem 3.3]{kemper01}, since 7 divides $|G|$ just once. Thus, a prime can divide $[G:N]$ without being bad. (By a computer calculation, no example of this phenomenon occurs below degree 7.)

At the other extreme, one might hope that the bad primes of $G$ are {\em only} those which are furnished by proposition \ref{prop:onlyifconstructive}. This is not true either. Take $G= D_7$, the dihedral group of order 14 acting on the vertices of a heptagon, which is also a Frobenius group. Now, all the nontrivial elements are 7-cycles and triple transpositions, so again, $N$ is trivial, and the candidate bad primes are 2 and 7. This time, they both really are bad primes. One can see this using Kemper's criterion \cite[Theorem 3.3]{kemper01}. For 2 it also follows from proposition \ref{prop:onlyifconstructive}, but for 7 it does not, since the 7-cycles have orbits that are properly refined by the triple transpositions.

Thus it remains to be determined, for a given $G$, exactly which primes are bad. Theorem \ref{mainresult} gives us a finite list of candidate bad primes (those dividing $[G:N]$), and, if this list is nonempty, proposition \ref{prop:onlyifconstructive} gives us some specific primes that are definitely bad. Among the remaining candidate bad primes, if any divide $|G|$ only once, \cite[Theorem 3.3]{kemper01} can be used to determine if they are actually bad. What remains to be determined is whether $p$ is a bad prime if $p^2\mid |G|$ and $p$ is not associated to a $g\in G\setminus N$ with minimal orbits as in proposition \ref{prop:onlyifconstructive}.

\begin{question}\label{q:whichpexactly}
How can Cohen-Macaulayness of $k[\mathbf{x}]^G$ be assessed when \cite[Theorem 3.3]{kemper01} and the present work are both inapplicable, i.e. when $p\mid [G:N]$ and $p^2 \mid |G|$, but $p$ does not come from a minimal $g\in G\setminus N$ as in proposition \ref{prop:onlyifconstructive}?
\end{question}

Another line of inquiry that flows from the present work has to do with the relationship between the arguments in the ``if" and ``only-if" directions. The proof of the ``if" direction is a mildly revised version of an argument given by the first author in his doctoral thesis \cite{blumsmith}. In that same work, he also proved the ``only-if" direction for $k[\Delta]^G$ (see \S \ref{sec:CCAandPL} for notation), but not for $k[\mathbf{x}]^G$. There, the ``only-if" argument was framed in the same topological language as the ``if" argument, which is why it applied to $k[\Delta]^G$ (taking advantage of Stanley-Reisner theory) but not $k[\mathbf{x}]^G$. The second author suggested to transfer the ``only-if" argument from topological into commutative-algebraic language, and much of the present paper sprang from this suggestion. 

This transfer was accomplished piecemeal, with an individual search for each com\-mutative-algebraic fact needed to replace each topological fact. For example, Raynaud's theorem (lemma \ref{ringslice}) replaced an elementary principle about the relationship between point stabilizers and the local structure in a topological quotient. The well-behavedness of inertia groups with respect to normal subgroups (lemma \ref{lem:GNinertia}) replaced an elementary fact about group actions on a set. The observation that inertia $p$-groups obstruct Cohen-Macaulayness if they are not generated by $2$-reflections (lemma \ref{lem:GKLP}), based on Lorenz and Pathak's \cite[Corollary 4.3]{lorenzpathak}, replaced an argument about the homology of links in the quotient of a simplicial complex. 

Nonetheless, the authors had the conviction throughout that an overarching principle was at play. It may be fruitful to seek a more comprehensive understanding of the relationship between the topology and the algebra. Stanley-Reisner theory gives a partial answer to this question, but it does not appear to account for the ``only-if" direction of theorem \ref{mainresult}, so a fuller picture is desirable.

Here are two more focused questions that approach this inquiry from various directions:

\begin{question}
Is there a purely algebraic proof of theorem \ref{mainresult}, making no use of Stanley-Reisner theory or Lange's result on PL manifolds?
\end{question}

\begin{question}\label{q:bidirectional}
For a fixed $p=\Char k$ as in question \ref{q:whichpexactly}, can $k[\mathbf{x}]^G$ be Cohen-Macaulay without $k[\Delta / G]$ being Cohen-Macaulay?
\end{question}


\section{Acknowledgement}

The authors wish to thank Gregor Kemper, Victor Reiner, and Christian Lange for fruitful discussions and comments, Tom Zaslavsky for alerting them to the conventional partial order on $\Pi_n$, and an anonymous referee for very helpful comments.

The second author's work on this project was partially supported by a grant from the National Research Foundation of South Africa.

Parts of this work originally appeared as part the first author's doctoral thesis under the supervision of Yuri Tschinkel and Fedor Bogomolov at the Courant Institute of Mathematical Sciences at NYU.

\end{document}